\documentclass[11pt]{article}

\usepackage[margin=1.08in]{geometry}
\usepackage{amsmath,amssymb,amsthm,mathtools, xcolor}
\usepackage{microtype}
\usepackage{enumitem}
\usepackage{comment}

\usepackage[hidelinks]{hyperref}
\hypersetup{
  pdftitle={Log-periodic asymptotics for tensor powers of the natural SL2-module in odd characteristic},
  pdfauthor={Aranya Lahiri and Nai-Heng Sheu}
}
\usepackage[backend=biber,style=alphabetic,maxbibnames=99]{biblatex}
\numberwithin{equation}{section}

\newtheorem{theorem}[equation]{Theorem}
\newtheorem{proposition}[equation]{Proposition}
\newtheorem{lemma}[equation]{Lemma}
\newtheorem{corollary}[equation]{Corollary}
\theoremstyle{definition}

\newtheorem{remark}[equation]{Remark}

\newcommand{\R}{\mathbb{R}}
\newcommand{\C}{\mathbb{C}}
\newcommand{\Z}{\mathbb{Z}}
\newcommand{\N}{\mathbb{N}}
\newcommand{\dd}{\,\mathrm{d}}

\newcommand{\SL}{\mathrm{SL}}
\newcommand{\Arg}{\operatorname{Arg}}

\newcommand{\arcosh}{\operatorname{arcosh}}

\newcommand{\commuNHH}{\textcolor{brown}} 

\newcommand{\Scal}{\mathcal S}
\newcommand{\Log}{\text{Log }}

\title{Asymptotics for number of indecomposable components of tensor powers of the natural $\mathrm{SL}_2$-module in odd characteristic}
\author{
  Aranya Lahiri\\
  Department of Mathematics, Indiana University
  \and
  Nai-Heng Sheu\\
  Department of Mathematics, Indiana University
}
\date{}

\begin{document}
\maketitle
\begin{abstract}

Let $K$ be an algebraically closed field of odd characteristic $p$, let
$G=\SL_2(K)$, and let $V$ be the natural representation of $G$.  Let $b_k$
denote the number of $G$-indecomposable factors of $V^{\otimes k}$, counted
with multiplicity, and let
$\delta_p=1-\log_{p^2}\!\bigl(\tfrac{p+1}{2}\bigr)$.  Then there exists a
smooth, strictly positive, multiplicatively $p^2$-periodic function
$\omega(t)$ such that $b_k$ is asymptotic to $\omega(k)k^{-\delta_p}2^k$.  We
also show that $t^{-\delta}\omega(t)$ arises as the limiting density of renormalized convolutions of rescaled copies of a positive weight $3/2$
theta function, obtained from the boundary heat flux of the Dirichlet heat
kernel on $(0,p)$.
\end{abstract}

\tableofcontents

\section{Introduction}

Let $K$ be an algebraically closed field and let $G=\SL_2(K)$.  For each
$n\geq0$ there is, up to isomorphism, a unique indecomposable tilting
$G$-module $T(n)$ of highest weight $n$; in particular, the natural
two-dimensional representation is $V=T(1)$.  Since tensor products of
tilting modules are again tilting, for every $k\geq0$ there are uniquely
determined multiplicities $a_{n,k}\geq0$ such that

$$
V^{\otimes k}\cong\bigoplus_{n\geq0}T(n)^{\oplus a_{n,k}}.
$$

We study the total number of indecomposable summands, counted
with multiplicity,

$$
b_k:=\sum_{n\geq0}a_{n,k},
$$
and in particular in the asymptotic behavior of $b_k$ as $k\to\infty$.

In characteristic zero the category of
finite-dimensional rational $G$-modules is semisimple, so the indecomposable
tilting modules are precisely the irreducible modules.  The Clebsch--Gordan
rule

$$
T(n)\otimes V\cong T(n+1)\oplus T(n-1),\qquad n\geq1,
$$
together with $T(0)\otimes V\cong T(1)$, identifies $a_{n,k}$ with the number of $\pm1$ walks of length $k$ from
$0$ to $n$ that remain in $\Z_{\geq0}$, the constraint at the origin
coming from $T(0)\otimes V\cong T(1)$.  Hence $b_k$ counts all such walks
of length $k$, and
\[
b_k=\binom{k}{\lfloor k/2\rfloor},
\qquad\text{so}\qquad
2^{-k}k^{1/2}b_k\longrightarrow\sqrt{\tfrac2\pi}.
\]

In characteristic $p$ semisimplicity fails, and the formula of the decomposition of
$T(n)\otimes V$ is determined by the base-$p$
expansion of $n$.  The effect on $b_k$ is twofold, and it is the second
part that concerns us.  First, the power of $k$ changes.  Second, and less
obviously, the normalized sequence no longer converges: it remains
bounded between two positive constants, but oscillates indefinitely, and
the oscillation does not decay.  Determining the power of $k$ is therefore
only half of the problem; the other half is to investigate a function that
the oscillation approaches.

\subsection*{Previous results}

Erdmann (\cite{Erdmann95}) translated a classical question of finding certain decomposition numbers of symmetric groups to a question about finding explicity formula for $a_{n,k}$ via Schur-Weyl duality. Then she considered the generating functions,

$$
X_n(t):=\sum_{k\geq0}a_{n,k}t^k.
$$
and showed that the $X_n(t)$ are rational functions expressable explicitly via
 Chebyshev polynomials, and obtained the corresponding explicit formulae for $a_{n,k}$.
Thus the exact multiplicities underlying the present counting problem were
already accessible well before their large-$k$ behavior became the object of
study.

The systematic asymptotic theory began with Coulembier, Ostrik, and
Tubbenhauer \cite{COT24}.  In a much broader tensor-categorical setting they
showed that the exponential growth rate of the number of indecomposable
summands is determined by the dimension of the object.  For the natural
$\SL_2$-module this gives
$$
\lim_{k\to\infty}b_k^{1/k}=2.
$$For the the subexponential
factor, Coulembier, Etingof, Ostrik, and Tubbenhauer subsequently determined its
power-law order in every positive characteristic \cite{CEOT25}.  They showed, if$$
\delta_p:=1-\log_{p^2}\!\left(\frac{p+1}{2}\right),
$$then there are constants $c_p,C_p>0$ such that
\begin{equation}\label{eq:intro-CEOT-bounds}
c_p \,2^k k^{-\delta_p}
\leq b_k\leq
C_p\,2^k k^{-\delta_p}.
\end{equation}
 One checks that $\delta_p>1/2$ for every $p$, and $\delta_p$ decreases to $\frac12$ as $p\to\infty$.  Thus positive characteristic produces fewer
indecomposable summands than characteristic zero, reflecting the larger
size of the indecomposable tilting modules, and the characteristic-zero
exponent is recovered in the limit.  In characteristic $2$ the same
authors showed in addition that the normalized sequence is controlled by
a continuous multiplicatively $4$-periodic function, so that the
two-sided bounds \eqref{eq:intro-CEOT-bounds} conceal a finer phenomenon.

Larsen made this phenomenon explicit in characteristic $2$
\cite{Larsen25}.  He showed that there is a positive smooth multiplicatively $4$-periodic
function $\omega_2$ such that $$
b_{2k}=b_{2k+1}\sim 4^k k^{-\delta_2}\omega_2(k).
$$His method was to consider the fusion graph of $V^{\otimes 2}$ and interpreted $a'_{n,k}:=a_{2n,2k}$ as the number of length-$k$ paths in this graph from $0$ to $n$. He then established the multiplicative property of $X_n$ with respect to the base-$2$ expansion of $n$. More precisely, if $n=\sum_{i=1}^m 2^{s_i},$ this multiplicative property expresses $a'_{n,k}$ as a finite convolution
of the sequences $a'_{2^{s_i},k}$ and it suggests that $\omega_2$ should arise as a limit of suitably rescaled convolutions. 
A particularly striking feature of his argument is that  $\omega_2(x)$ is
described by an infinite convolution of copies of a weight $\frac{3}{2}$ theta function that he can predict already by explicitly writing down the coefficients of the generating functions $X_{2^s}(t)$ using Chebyshev polynomials.

For odd characteristic the second author established two-sided bounds of
the form \eqref{eq:intro-CEOT-bounds} for an arbitrary tilting module
$T$, with the factor $2$ replaced by $\dim T$ \cite{Sheu25}.  That
argument builds on the framework of \cite{Larsen25} together with the
bounds \eqref{eq:intro-CEOT-bounds} of \cite{CEOT25}.  It also provides
explicit rational formulas for the generating functions $X_n(t)$ and
shows that, after a convenient shift and normalization, the resulting
family is multiplicative with respect to the base-$p$ expansion of $n$.

Our purpose is to close the gap in odd characteristic and go  from a two-sided estimate
to an asymptotic like Larsen does in characteristic 2.

\begin{theorem}[Main theorem]\label{thm:main}
Let $p$ be an odd prime and let
$\delta_p=1-\log_{p^2}\!\bigl(\tfrac{p+1}{2}\bigr)$ be the exponent of
\eqref{eq:intro-CEOT-bounds}.  There is a smooth function
$\omega:(0,\infty)\to(0,\infty)$ with $\omega(p^2t)=\omega(t)$ for all
$t>0$, such that
\[
b_k\sim 2^kk^{-\delta_p}\omega(k).
\]
\end{theorem}

We in fact prove the sharper statement $b_k=2^kk^{-\delta_p}\omega(k)\bigl(1+O(k^{-\frac{1}{2}})\bigr)$
in Proposition~\ref{prop:additive-asymptotic}.
The formula for the exponent places the result naturally between the two
cases discussed above.  At $p=2$ it gives Larsen's exponent $1-\log_4\!\left(\frac32\right)$, while
$\delta_p\longrightarrow\frac12$ as $p\to\infty,$ recovering the characteristic-zero power $k^{-1/2}$.

The main assertion of the theorem is the existence of $\omega$, and part of
what we prove is that $\omega$ admits two independent descriptions.  The
first is an explicit Fourier series
\[
\omega(t)=\sum_{m\in\Z}\frac{d_m}{\Gamma(\beta_m)}\,t^{i\tau_m},
\qquad
\tau_m=\frac{2\pi m}{\log(p^2)},
\quad
\beta_m=(1-\delta)+i\tau_m,
\]
whose coefficients decay exponentially. The second is realization of it as an infinite convolution of a theta function with proper scaling.  Let $K_p(x,y,s)$ be the Dirichlet heat kernel on the interval
$(0,p)$ and put
\[
\varphi_p(t):=\frac14\sum_{a=1}^{p-1}\partial_yK_p(a,0,t/4)
\qquad(t>0),
\]
the sum of the inward heat fluxes at the endpoint $0$ produced by unit
sources at the nonzero integral points of $(0,p)$. We show that $t^{-\delta}\omega(t)$ is the limiting density of
convolutions of rescaled copies of $\varphi_p$. Thus, Larsen's
theta-convolution phenomenon therefore persists in odd characteristic.

\subsection*{Method and Organization}

The results quoted above were obtained by different analytic means:
\cite{CEOT25} uses self-similar functional equations together with
real-variable and Tauberian estimates, while \cite{Larsen25} works from
explicit coefficient formulas and a limiting convolution construction.
We proceed instead through the complex singularities of a 
generating function, using the harmonic-sum formalism of Flajolet,
Gourdon and Dumas \cite{FlGoDu95} and the singularity analysis of
Flajolet and Sedgewick \cite{FlSe09}.

\vskip8pt
\emph{From tilting multiplicities to an infinite product
(Section~\ref{sec:product}).}  Since $b_k=\sum_na_{n,k}$, the counting
series for $b_k$ is obtained by summing the multiplicity generating
functions $X_n$ over all highest weights, and normalizing by
$\dim V^{\otimes k}=2^k$ amounts to evaluating at $t=z/2$:
\[
\mathcal B(z):=\sum_{n\geq0}X_n(z/2)=\sum_{k\geq0}2^{-k}b_kz^k.
\]
With this scaling the radius of convergence is $1$, so all subexponential
information about $b_k$ lies in the behavior of $\mathcal B$ on the unit
circle.  It is the shifted series
\[
\mathcal G(z):=1+\frac z2\mathcal B(z)
\]
that inherits multiplicativity: the base-$p$ factorization of
\cite{Sheu25} converts the sum over all highest weights into a product
over the base-$p$ digit positions.  In the variable
$u=u(z)=\arcosh(1/z)$, for $0<z<1$, this reads
\begin{equation}\label{eq:intro-product}
\mathcal G(z)=\prod_{s\geq0}\bigl(1+h_p(p^su)\bigr),
\qquad
h_p(y):=\sum_{j=1}^{p-1}\frac{\sinh(jy)}{\sinh(py)}.
\end{equation}

\vskip8pt

\emph{Mellin analysis (Section~\ref{sec:mellin}).}  In the local variable
$x=u^2/4$ at $z=1$ the product \eqref{eq:intro-product} becomes
self-similar under $x\mapsto p^2x$, and taking logarithms turns \eqref{eq:intro-product} into a
harmonic sum over a geometric progression of scales.  Its Mellin transform is the product of the transform of
$\log\bigl(1+h_p\bigr)$ with the factor $(1-p^{-s})^{-1}$ coming from the
sum over scales. The second factor has poles along an arithmetic progression on
the imaginary axis: the pole at the origin produces the power law, and
the remaining ones produce its Fourier corrections.  Near $z=1$ this
gives a singular expansion of $\mathcal B$ of the form
\[
(1-z)^{-\alpha}\times\Bigl(\text{a $1$-periodic function of }
\log_{p^2}\tfrac1{1-z}\Bigr),
\qquad
\alpha=\log_{p^2}\!\left(\frac{p+1}{2}\right)=1-\delta_p,
\]
the exponent arising because one step of the scaling multiplies the local
variable by $p^2$ while contributing a factor $\frac{p+1}{2}$ at the
origin.  The Fourier coefficients of the periodic factor are the residues
at the nonzero poles, and they decay exponentially; this decay is what
makes the coefficient extraction of the next step possible. 

\vskip8pt
\emph{Smooth amplitude (Section~\ref{sec:transfer}).}  The asymptotics of the coefficients of the infinite product $\mathcal{G}(z)$ are determined by its boundary singularities at $z=1$ and at $z=-1$. The singularity at $z=1$ yields the smooth amplitude $\omega(t)$, which we construct in two steps. First, we define a periodic singular model as an infinite sum of $(1-z)^{-\beta_m}$ over $m$ at $z=1$ via the Fourier series from Section~\ref{sec:mellin}. The approximation of $(1-z)^{-\beta_m}$ in terms of $\Gamma(\beta_m)$ motivates the defining series for $\omega(t)$. At $z=-1$, the $O$-transfer method shows that the singularity at $z=-1$ contributes the error terms.

\vskip8pt
\emph{The modular picture (Sections~\ref{sec:theta} and
\ref{sec:convolution}).}  In Section~\ref{sec:convolution}, we give a realization of $\omega(t)$ as the smooth density of the limit measure of finite convolutions of the rescaled measure $\delta_0+\varphi_p(t)\dd t$, and show that the strict positivity of $\omega$ follows from that of $\varphi_p$, which is provided by the parabolic Hopf lemma in Section~\ref{sec:theta}. We also show that $\varphi_p$ is a theta function of weight $3/2$ by Poisson summation in Section~\ref{sec:theta}.

\vskip8pt
The method is not specific to $V$, adopting inputs from \cite{Sheu25} we should be able to give the asymptotics for the number of  indecomposable summands of $T^{\otimes k}$ for any tilting summand $T$. We intend to add this result in a future version.

\subsection{Glossary of terms}

\begin{itemize}[leftmargin=3.1em,style=nextline,itemsep=0.25em]
\item[$p,\ K,\ G,\ V$] an odd prime, an algebraically closed field of
characteristic $p$, the group $\SL_2(K)$, and the natural module $V=T(1)$.
\item[$T(n)$] the indecomposable tilting $G$-module of highest weight $n$.
\item[$a_{n,k},\ b_k$] the multiplicity of $T(n)$ in $V^{\otimes k}$, and
$b_k=\sum_{n\geq0}a_{n,k}$.
\item[$X_n(t)$] the multiplicity generating function
$\sum_{k\geq0}a_{n,k}t^k$.
\item[$Z_n,\ W_n$] the shifted and normalized versions of $X_n$:
$Z_n=X_{n-1}$ for $n\geq1$, $Z_0=1/t$, and $W_n=tZ_n$, so $W_0=1$.
\item[$\mathcal B(z)$] the normalized counting series
$\sum_{k\geq0}2^{-k}b_kz^k=\sum_{n\geq0}X_n(z/2)$.
\item[$\mathcal G(z)$] the shifted series
$\sum_{n\geq0}W_n(z/2)=1+\tfrac z2\mathcal B(z)$.
\item[$T_n,U_n,Q_n,R_m$] Chebyshev polynomials of the first and second
kinds and their rescalings $Q_n(x)=2T_n(x/2)$, $R_m(x)=U_m(x/2)$.
\item[$S_s,\ A_s$] the digit-position factors of $\mathcal G$:
$S_s(t)=\sum_{a=0}^{p-1}W_{ap^s}(t)$, and $A_s$ its expression in the
coordinate $q$.
\item[$\Lambda$] the mass ratio $\Lambda=(p+1)/2=1+h_p(0)$.
\item[$\alpha,\ \delta$] the exponents $\alpha=\log_{p^2}\Lambda$ and
$\delta=1-\alpha$.
\item[$h_p(y)$] the single-scale factor
$\sum_{j=1}^{p-1}\sinh(jy)/\sinh(py)$.
\item[$v(z),\ u(z),\ q(z)$] the hyperbolic variable
$v(z)=\arcosh(1/z)$ on $(0,1)$, its continuation $u(z)=-\log q(z)$ near
$z=1$, and the Joukowski coordinate $q(z)=z/(1+\sqrt{1-z^2})$.
\item[$\eta(z)$] the local variable $u(z)^2/4$ at $z=1$; it satisfies
$\eta(z)=(1-z)/2+O((1-z)^2)$.
\item[$\mu(z)$] the local variable $\arcosh(-1/z)$ at $z=-1$.
\item[$\Omega_z,\ \Delta_{\phi,\rho}$] the slit plane
$\C\setminus((-\infty,-1]\cup[1,\infty))$ and the doubly indented disk
used for transfer.
\item[$f_p,\ F_p$] $f_p(x)=h_p(2\sqrt x)$ and
$F_p(x)=\prod_{s\geq0}(1+f_p(p^{2s}x))$, so
$\mathcal G(z)=F_p(\eta(z))$ near $z=1$.
\item[$g_p,\ H_p,\ S_p$] $g_p(y)=\log(1+h_p(y))$,
$H_p(y)=\prod_{s\geq0}(1+h_p(p^sy))$, and the harmonic sum
$S_p(y)=\sum_{k\geq0}g_p(p^ky)=\log H_p(y)$.
\item[$f^*(s)$] the Mellin transform $\int_0^\infty f(x)x^{s-1}\dd x$.
\item[$a_0$] the constant $\log\Lambda=2\alpha\log p$.
\item[$\widetilde\Phi,\widetilde\Omega,\Omega,\Omega_B$] the $1$-periodic
functions produced by the Mellin poles, with
$\widetilde\Omega=e^{\widetilde\Phi}$ and $\Omega,\Omega_B$ its
normalizations.
\item[$\tau_m,\ \beta_m$] the frequencies $\tau_m=2\pi m/\log(p^2)$ and
the exponents $\beta_m=\alpha+i\tau_m$.
\item[$c_m,\ d_m$] the Fourier coefficients of $\Omega$ and of
$\Omega_B$, related by $d_m=2^{1+\beta_m}c_m$.
\item[$\mathcal S(z)$] the singular model $\sum_m d_m(1-z)^{-\beta_m}$ at
$z=1$.
\item[$\omega(t)$] the periodic amplitude
$\sum_m\Gamma(\beta_m)^{-1}d_mt^{i\tau_m}$, with $\omega(p^2t)=\omega(t)$.
\item[$M_p,\ J_p,\ \gamma$] the local factor at $z=-1$, the product
$J_p(v)=\prod_{s\geq0}M_p(p^sv)$, and $\gamma=1-2\alpha$.
\item[$K_p(x,y,s)$] the Dirichlet heat kernel on $(0,p)$.
\item[$\varphi_p(t)$] the boundary heat-flux kernel
$\tfrac14\sum_{a=1}^{p-1}\partial_yK_p(a,0,t/4)$, with
$\int_0^\infty\varphi_p=f_p(0)=(p-1)/2$.
\item[$\varphi_{p,s},\ \Psi_{p,I},\ \psi^{(0)}_{p,r},\ \psi_{p,r}$] the
rescaled kernels, their finite convolution sums, and the renormalized
densities.
\item[$\chi,\ \nu,\ \psi$] the rescaled kernel $\chi(t)=2\varphi_p(2t)$,
the cofactor measure in the positivity argument, and the limiting density
$\psi(t)=t^{-\delta}\omega(t)$.
\end{itemize}

\subsection{AI Usage}

ChatGPT (principally GPT-5.5) was used in the development of some of the analytic arguments in this paper. The authors checked and rewrote the arguments and are responsible for the final proofs and exposition.

Here is an account, to the best of our recollection, of how ChatGPT was used in developing the proof.

\begin{enumerate}
\item We first worked out, essentially modifying Larsen's characteristic $2$ methods and the generating function formulas available in odd characteristic, what the asymptotic behavior should be. Some parts of Larsen's analysis appeared to us rather ingenious, and we wanted a more systematic analytic framework for obtaining the asymptotic in odd characteristic.

\item This desire led us to explore harmonic-sum method of Flajolet--Gourdon--Dumas and the singularity-analysis methods of Flajolet--Sedgewick as a promising framework. After doing some preliminary calculations and being convinced that the framework does explain the log periodicty, the first author asked ChatGPT to carry out the explicit Mellin transform for the characteristic 2 case. After substantial back and forth, and lots of corrections of errors he was convinced that GPT can supply a proof of Proposition \ref{prop:F-periodic}. GPT suggested using the Joukowski coordinates, and also helped with many small details. For example, the first author prompted it to use Stirling's formula for the estimations, the proof of identities (2.15) and (2.16) of $h_p(y)$, etc.

\item For the singularity at $z=-1$ the first author knew that we wanted to show that it can be ignored. He prompted GPT to supply a proof and it supplied a very terse proof. While reading through it, we added several explanatory steps and also arranged it so that at least we can read it.

\item We already knew what $\varphi_p$ should be quite explicitly by following the same technique as Larsen in characteristic 2. We also observed that $\varphi_p$ should be related to $f_p$ via Laplace transform. But did not initially know the appropriate analytic realization or how best to formalize Larsen's machinery in our setting. The first author gave a prompt to ChatGPT asking how to get the similar results as Larsen in this new framework. GPT came up with heat-kernel interpretation in Section 5, in which $\varphi_p$ is realized as a sum of boundary heat fluxes for the Dirichlet heat kernel and the Laplace-transform identity is proved using the Dirichlet resolvent.

\item ChatGPT also helped make the limiting-convolution argument in Section 6 precise. In particular, it assisted with the formulation of the finite convolution measures, the passage to a limiting positive measure, the identification of its density through Laplace transforms, and the argument upgrading nonnegativity of the limiting density to strict positivity. 

\item During the preparation of the manuscript, ChatGPT was also used to check proofs, suggest corrections and reorganizations, make the glossary of notation and help with literature search. The authors independently checked the mathematical arguments and rewrote the material generated by GPT as they found the GPT generated proofs very hard to read. It is overall fair to say that GPT made a lot of mistakes at several points that needed correction, however, without GPT's assistance it would have taken the authors much longer to discover some of the trickier analytic arguments as used in the paper.
\end{enumerate}

\subsection{Acknowledgments} We would like to thank Michael Larsen for suggesting the problem and all the productive discussions at various points of the project.
\section{From tilting multiplicities to a discrete-scale product}\label{sec:product}

\subsection[Generating functions and the base-$p$ factorization]{Generating functions and the base-$p$ factorization}

For $n,k\geq0$, let $a_{n,k}$ be the multiplicity of $T(n)$ in
$V^{\otimes k}$.  Thus the total number of indecomposable summands of
$V^{\otimes k}$, counted with multiplicity, is
\begin{equation}\label{eq:def-bk-section2}
 b_k:=\sum_{n\geq0}a_{n,k}.
\end{equation}
For each $n\geq0$, define the generating function
\[
X_n(t):=\sum_{k\geq0}a_{n,k}t^k.
\]
Following \cite{Sheu25}, set
\[
Z_n(t):=X_{n-1}(t)\quad(n\geq1),
\qquad
Z_0(t):=\frac1t,
\qquad
W_n(t):=tZ_n(t), \qquad W_0=1.
\]
The formula of $W_{ap^s(t)}$ was given in terms of the Chebyshev polynomials in \cite{Erdmann95} and \cite{Sheu25}. Therefore, we recall the notion of the Chebyshev polynomials. 

Let $T_n$ and $U_n$ denote the Chebyshev
polynomials of the first and second kinds, characterized by
\[
T_n(\cos\theta)=\cos(n\theta),
\qquad
U_n(\cos\theta)=\frac{\sin((n+1)\theta)}{\sin\theta}.
\]
Define the rescaled polynomials $Q_n(x):=2T_n(x/2)$ and $R_m(x):=U_m(x/2).$
Then
\[
Q_n(2\cos\theta)=2\cos(n\theta),
\qquad
R_m(2\cos\theta)=\frac{\sin((m+1)\theta)}{\sin\theta}.
\]

\begin{proposition}\label{prop:structural-input}
For $s\geq0$, $0\leq a<p$, and $0\leq i<p^s$,
\begin{equation}\label{eq:W-multiplicative}
W_{ap^s+i}(t)=W_{ap^s}(t)W_i(t).
\end{equation}
Moreover,
\begin{equation}\label{eq:W-block-rational}
W_{ap^s}(t)=
\frac{R_{p-a-1}\bigl(Q_{p^s}(1/t)\bigr)}
     {R_{p-1}\bigl(Q_{p^s}(1/t)\bigr)}.
\end{equation}
\end{proposition}

\begin{proof}
    See \cite[Proposition~3.3 and the proof of
Proposition~3.10]{Sheu25}.  
\end{proof}
We normalize the counting sequence by the exponential growth of
$\dim(V^{\otimes k})=2^k$ and write
\[
\mathcal B(z):=\sum_{k\geq0}2^{-k}b_k z^k,
\qquad
\mathcal G(z):=\sum_{n\geq0}W_n(z/2).
\]

\begin{lemma}\label{lem:B-via-G}
The generating functions $\mathcal B$ and $\mathcal G$ satisfy
\begin{equation}\label{eq:B-via-G}
\mathcal G(z)=1+\frac z2\mathcal B(z).
\end{equation}
Equivalently, for $z\neq0$,
\begin{equation}\label{eq:B-from-G}
\mathcal B(z)=\frac{2}{z}\bigl(\mathcal G(z)-1\bigr).
\end{equation}
The right-hand side of \eqref{eq:B-from-G} extends analytically to $z=0$.
\end{lemma}

\begin{proof}
For $n\geq1$ we have $W_n(t)=tX_{n-1}(t)$, and hence
\[
\sum_{n\geq1}W_n(t)
=t\sum_{m\geq0}X_m(t)
=t\sum_{k\geq0}\left(\sum_{m\geq0}a_{m,k}\right)t^k
=t\sum_{k\geq0}b_kt^k.
\]
Substituting $t=z/2$ and adding the term $W_0=1$ gives
\eqref{eq:B-via-G}; hence, \eqref{eq:B-from-G} follows.  Since \eqref{eq:B-via-G} shows that
$\mathcal G(z)-1$ is divisible by $z$ as a power series, the quotient in
\eqref{eq:B-from-G} has a well-defined analytic value at $z=0$.
\end{proof}

For each scale $s\geq0$, define
\[
S_s(t):=\sum_{a=0}^{p-1}W_{ap^s}(t).
\]
The multiplicative identity \eqref{eq:W-multiplicative} now turns the
base-$p$ expansion of $n$ into a product over scales.

\begin{proposition}\label{prop:digit-product}
For every $r\geq0$,
\begin{equation}\label{eq:finite-digit-product}
\sum_{0\leq n<p^{r+1}}W_n(t)=\prod_{s=0}^{r}S_s(t).
\end{equation}
Consequently,
\begin{equation}\label{eq:G-formal-product}
\mathcal G(z)=\prod_{s\geq0}S_s(z/2)
\end{equation}
as an identity of formal power series.
\end{proposition}

\begin{proof}
Let
\[
n=a_0+a_1p+\cdots+a_rp^r,
\qquad 0\leq a_s<p,
\]
be the base-$p$ expansion of $n$.  Repeated application of
\eqref{eq:W-multiplicative} gives
\[
W_n(t)=\prod_{s=0}^{r}W_{a_sp^s}(t).
\]
Summing independently over the digits $a_0,\ldots,a_r$ gives
\eqref{eq:finite-digit-product}.

It remains to explain why the limit of these finite products gives
\eqref{eq:G-formal-product} coefficientwise.  For $n\geq1$,
\[
W_n(t)=tX_{n-1}(t)
      =\sum_{k\geq0}a_{n-1,k}\,t^{k+1}.
\]
If $a_{n-1,k}\neq0$, then $T(n-1)$ occurs in $V^{\otimes k}$.  Since every
highest weight occurring in $V^{\otimes k}$ is at most $k$, this is impossible
when $n-1>k$.  Therefore the coefficient of $t^j$ in $W_n(t)$ is zero whenever
$n>j$.

Fix $j\geq0$.  It follows that the coefficient of $t^j$ in
$\sum_{n\geq0}W_n(t)$ receives contributions only from the finitely many
terms $W_0,\ldots,W_j$.  Hence, once $p^{r+1}>j$, the coefficient of $t^j$ in
\[
\sum_{0\leq n<p^{r+1}}W_n(t)
\]
has already reached its final value.  By \eqref{eq:finite-digit-product}, the
same is therefore true of the coefficient of $t^j$ in the corresponding
finite product.  Since this holds for every $j$, the finite products converge
coefficientwise to \eqref{eq:G-formal-product}.
\end{proof}

For the later analytic argument, we introduce the variable
\[
v(z):=\operatorname{arcosh}(1/z), \qquad 0<z<1,
\]
where $\operatorname{arcosh}$ denotes the positive real branch.
It is then convenient to write the Chebyshev identities in hyperbolic
form:
\begin{equation}\label{eq:R-hyperbolic}
Q_n(2\cosh y)=2\cosh(ny),
\qquad
R_m(2\cosh y)=\frac{\sinh((m+1)y)}{\sinh y}.
\end{equation}

Then we have
\begin{proposition}\label{prop:one-scale}
For $0<z<1$, $s\geq0$, and $0\leq a<p$,
\begin{equation}\label{eq:W-hyperbolic}
W_{ap^s}(z/2)=
\frac{\sinh((p-a)p^sv(z))}{\sinh(p^{s+1}v(z))}.
\end{equation}
Consequently,
\begin{equation}\label{eq:S-hyperbolic}
S_s(z/2)=1+h_p(p^sv(z)),
\qquad
h_p(y):=\sum_{j=1}^{p-1}\frac{\sinh(jy)}{\sinh(py)}.
\end{equation}
Moreover,
\begin{align}
 h_p(y)
 &=\frac{\sinh((p-1)y/2)}{2\sinh(y/2)\cosh(py/2)}
   =\frac{e^{-y}+\cdots+e^{-(p-1)y}}{1+e^{-py}},
   \label{eq:hp-forms}\\
1+h_p(y)
 &=\frac{1-e^{-(p+1)y}}{(1-e^{-y})(1+e^{-py})}.
   \label{eq:one-plus-hp}
\end{align}
\end{proposition}

\begin{proof}
With $t=z/2$ we have $1/t=2/z=2\cosh v(z)$.  Applying
\eqref{eq:R-hyperbolic} to \eqref{eq:W-block-rational} therefore gives
\eqref{eq:W-hyperbolic}.  Summing over $a$ and reindexing with $j=p-a$
gives \eqref{eq:S-hyperbolic}.

To prove equation \eqref{eq:hp-forms}, first, observe that $1-e^{-y}=2e^{-y/2}\sinh{(y/2)}.$ Then finite geometric-series formula gives
\[
\sum_{j=1}^{n-1}e^{\pm jy}
=
e^{\pm ny/2}
\frac{\sinh((n-1)y/2)}{\sinh(y/2)}.
\]
Then \eqref{eq:hp-forms} follows by applying the equation $\sinh(py)=2\sinh(py/2)\cosh(py/2)$ and the above formula. Equation \eqref{eq:one-plus-hp} then follows from \eqref{eq:hp-forms}.
\end{proof}

\begin{lemma}\label{lem:radius-one}
The radius of convergence of $\mathcal B$ is $1$, and $z=1$ is a boundary
singularity of $\mathcal B$.
\end{lemma}

\begin{proof}
Since $b_k$ is the number of indecomposable direct summands of the $2^k$-dimensional module
$V^{\otimes k}$, counted with multiplicity, we have $b_k\leq2^k$.  Hence, 
$0\le 2^{-k}b_k \leq1$, so the radius of convergence of $\mathcal B$ is at least
$1$.

To show that $z=1$ is a singularity of $\mathcal{B}(z)$, we first show that $\lim_{z \to 1}\mathcal{G}(z)= \infty.$

Observe that $W_n(z) \ge 0$ for all $z>0$ as the coefficients of $W_n(z)$ are nonnegative. Then for any $r \in \N$,
$$\mathcal{G}(z)= \sum_{n \ge 0} W_n(z/2) \ge \sum_{0 \le n < p^{r+1}} W_n(z/2)= \prod_{s=0}^rS_s(z/2), $$ where the last equality follows from Proposition~\ref{prop:digit-product}.

As $z\to 1$, we have $v(z)\to0$, and therefore, for each
fixed $s$,
\[
S_s(z/2)
=
1+h_p(p^sv(z))
\longrightarrow
1+h_p(0)
=
\frac{p+1}{2}.
\]
Hence, for any $r \in \N$,
$
\liminf_{z\to 1}\mathcal G(z)
\ge
(\frac{p+1}{2})^{r+1}.
$
As the inequality holds for every
$r$ and $p>2$,  the right-hand side tends to
infinity as $r\to\infty$. It follows that $z=1$ is a boundary singularity of $\mathcal{G}(z)$. By Lemma \ref{lem:B-via-G}, $z=1$ is also a boundary singularity of $\mathcal{B}(z)$, and the radius
of convergence of $\mathcal{B}(z)$ is $1$.
\end{proof}

\subsection{The Joukowski coordinate and analytic continuation}

The exponential form \eqref{eq:hp-forms} suggests replacing the 
variable $v=\arcosh (1/z)$, by
\[
q=e^{-v}.
\]
From $z=1/\cosh v$ we obtain
\[
z=\frac{2}{q^{-1}+q}=\frac{2q}{1+q^2}.
\]
Thus $q$ is naturally the inverse coordinate for the classical Joukowski map. 


Let
\[
\Omega_z:=\C\setminus\bigl(( -\infty,-1]\cup[1,\infty)\bigr)
\]
and choose the branch of $\sqrt{1-z^2}$ that is positive on $(-1,1)$.  Define
\begin{equation}\label{eq:def-q}
q(z):=\frac{z}{1+\sqrt{1-z^2}},
\qquad z\in\Omega_z.
\end{equation}

\begin{lemma}\label{lem:q-coordinate}
The function $q$ is analytic on $\Omega_z$, satisfies
\begin{equation}\label{eq:Joukowski-inverse}
z=\frac{2q(z)}{1+q(z)^2},
\end{equation}
and 
\begin{equation}\label{eq:q-unit-disc}
|q(z)|<1
\qquad (z\in\Omega_z).
\end{equation}
For $0<z<1$,
\begin{equation}\label{eq:q-exp-u}
q(z)=e^{-v(z)}.
\end{equation}
\end{lemma}

\begin{proof}
As the chosen square root is analytic on $\Omega_z$ and $1+\sqrt{1-z^2} \neq 0$ on $\Omega_z$, $q(z)$ is analytic on $\Omega_z$.  Starting from \eqref{eq:def-q} and solving algebraically for $z$ gives
\eqref{eq:Joukowski-inverse}.

We next prove \eqref{eq:q-unit-disc}.  First, $|q(z)|$ can never equal $1$ on
$\Omega_z$.  Indeed, if $q(z)=e^{i\theta}$, then
\[
z=\frac{2e^{i\theta}}{1+e^{2i\theta}}
 =\frac{1}{\cos\theta},
\]
which lies in $(-\infty,-1]\cup[1,\infty)$ (with the case
$\cos\theta=0$ corresponding to the point at infinity), contrary to
$z\in\Omega_z$. As $\Omega_z$ is connected, and $q$ is continuous with $q(0)=0$ and $|q(z)| \neq 1$ on $\Omega_z$, we have $|q(z)|<1$.


Finally, for $0<z<1$ we have $v=\arcosh(1/z)>0$. As $q(z)$ and $e^{-v}$ lie in $(0,1)$ and solve
\eqref{eq:Joukowski-inverse}, $q(z)=e^{-v(z)}$.
\end{proof}

For $s\geq0$ define
\begin{equation}\label{eq:def-As}
A_s(z):=
1+\frac{q(z)^{p^s}+q(z)^{2p^s}+\cdots+q(z)^{(p-1)p^s}}
        {1+q(z)^{p^{s+1}}}.
\end{equation}
Equivalently, summing the finite geometric series in the numerator gives
\begin{equation}\label{eq:As-product-form}
A_s(z)=
\frac{1-q(z)^{(p+1)p^s}}
     {(1-q(z)^{p^s})(1+q(z)^{p^{s+1}})}.
\end{equation}

\begin{proposition}\label{prop:global-continuation}
The product
\begin{equation} \label{eq:G-q-product}
    \mathcal G(z)=\prod_{s\ge0} A_s(z)
\end{equation}
converges normally on compact subsets of $\Omega_z$ and defines a
holomorphic extension of $\mathcal G$ from the unit disk to $\Omega_z$.
Consequently, $\mathcal B$ also extends holomorphically to $\Omega_z$.
In particular, the only possible singularities of $\mathcal B$ on the
unit circle are $z=1$ and $z=-1$.
\end{proposition}

\begin{proof}
For $0<z<1$, equations \eqref{eq:q-exp-u}, \eqref{eq:hp-forms}, and
\eqref{eq:S-hyperbolic} give
\[
A_s(z)=S_s(z/2).
\]
Hence, on this real interval, the product in \eqref{eq:G-q-product} agrees
with the formal product for $\mathcal G$.

Let $K\subset\Omega_z$ be compact.  By Lemma~\ref{lem:q-coordinate}, there is
$r_K<1$ such that
\[
|q(z)|\leq r_K
\qquad(z\in K).
\]
From the defining form \eqref{eq:def-As},
\begin{equation}\label{eq:As-minus-one-bound}
A_s(z)-1
=
\frac{q(z)^{p^s}+q(z)^{2p^s}+\cdots+q(z)^{(p-1)p^s}}
     {1+q(z)^{p^{s+1}}}.
\end{equation}
The numerator is bounded in absolute value by
\[
\sum_{j=1}^{p-1}r_K^{jp^s}
\leq (p-1)r_K^{p^s},
\]
while
\[
|1+q(z)^{p^{s+1}}|
\geq1-r_K^{p^{s+1}}
\geq1-r_K^p>0.
\]
Therefore
\begin{equation}\label{eq:As-minus-one-OK}
\sup_{z\in K}|A_s(z)-1|
\leq C_K r_K^{p^s}
\end{equation}
for a constant $C_K$ independent of $s$.  Since
\[
\sum_{s\geq0}r_K^{p^s}<\infty,
\]
the series $\sum_s\sup_K|A_s-1|$ converges. The standard convergence
criterion for infinite products of holomorphic functions therefore shows
that $\prod_sA_s$ converges uniformly on $K$ and defines a holomorphic
function there.  Since $K$ was arbitrary, the convergence is normal on
compact subsets of $\Omega_z$.

The resulting holomorphic function agrees with $\mathcal G$ on the interval
$0<z<1$, so the identity theorem identifies it with the analytic continuation
of $\mathcal G$.

Finally, \eqref{eq:B-from-G} defines the continuation of $\mathcal B$ for
$z\neq0$.  At $z=0$ there is no new singularity: on the original disk,
\eqref{eq:B-via-G} shows that $\mathcal G(z)-1$ is divisible by $z$, so the
quotient extends there with the original power-series value of
$\mathcal B$.  Since the slit domain excludes the rays beginning at $\pm1$,
the only points of the unit circle not covered by this continuation are
$z=1$ and $z=-1$.
\end{proof}

For the coefficient-transfer argument later in the paper, we shall work in a
domain that approaches both boundary points while avoiding the two slit rays.
Fix $0<\phi<\pi/2$ and $\rho>0$ and write
\begin{equation}\label{eq:def-double-delta}
\Delta_{\phi,\rho}:=
\left\{
\begin{array}{l}
|z|<1+\rho,\quad z\neq\pm1,\\[2mm]
|\Arg(z-1)|>\phi,\\[1mm]
|\Arg(z+1)|<\pi-\phi
\end{array}
\right\}.
\end{equation}
For any $\rho>0$, this open doubly indented disk is contained in
$\Omega_z$.

\subsection[The local variable at $z=1$]{The local variable at $z=1$}
Inside $\Delta_{\phi,\rho}$ near $z=1$, choose the
branch of $\log q(z)$ that tends to $0$ as $z\to1$ and set
\begin{equation}\label{eq:def-u-eta-local}
u(z):=-\log q(z),
\qquad
\eta(z):=\frac{u(z)^2}{4}.
\end{equation}
For $0<z<1$, $u(z)$ agrees with $u(z)=v(z)$. 

By the definition of $u(z)$, there exists a small open disk $B$ around $z=1$ so that $u(z)$ is analytic on $D=B \cap \Delta_{\phi,\rho}.$ Consequently, $\eta(z)$ is analytic on $D$. Note that, as a subset of $\Delta_{\phi,\rho}$, $1 \not \in D$.


\begin{lemma}\label{lem:eta-local}
The function $\eta$ extends analytically across $z=1$.  More precisely,
\begin{equation}\label{eq:eta-factor}
\eta(z)=(1-z)\Xi(z),
\qquad
\Xi(1)=\frac12,
\end{equation}
where $\Xi$ is analytic near $z=1$.  In particular,
\begin{equation}\label{eq:eta-asymp}
\eta(z)=\frac{1-z}{2}+O((1-z)^2).
\end{equation}
Moreover, 
\begin{equation}\label{eq:eta-sector}
|\arg\eta(z)|\leq\pi-\varepsilon
\end{equation}
for some $\varepsilon>0$ on a small indented neighborhood of $z=1$, i.e., on the intersection of $\Delta_{\phi,\rho}$ with a small neighborhood of $z=1.$
\end{lemma}

\begin{proof}
From $q=e^{-u}$ and \eqref{eq:Joukowski-inverse},
\[
z=\frac{2e^{-u}}{1+e^{-2u}}=\frac1{\cosh u}.
\]
The function $1/\cosh u$ is even in $u$, so it is analytic as a function of
$w=u^2$ near $w=0$. Its Taylor expansion is
\[
z=1-\frac{w}{2}+\frac{5w^2}{24}+O(w^3).
\]
The derivative of $z$ with respect to $w$ at $w=0$ is $-1/2\neq0$.  Hence the
analytic inverse function theorem gives an analytic function $w=w(z)$ near
$z=1$ satisfying
\[
w=2(1-z)+O((1-z)^2).
\]
Since $w=u^2$, we obtain
\[
\eta(z)=\frac{u^2}{4}
       =\frac{1-z}{2}+O((1-z)^2).
\]
Thus $\eta$ extends analytically across $z=1$, and
\[
\Xi(z):=\frac{\eta(z)}{1-z}
\]
also extends analytically there with $\Xi(1)=1/2$.  This proves
\eqref{eq:eta-factor} and \eqref{eq:eta-asymp}.

Finally, after shrinking the neighborhood, $\Xi(z)$ stays arbitrarily close
to the positive real number $1/2$.  The indentation keeps $1-z$ a fixed
positive angular distance away from the negative real axis.  Since
$\eta(z)=(1-z)\Xi(z)$, the same is true of $\eta(z)$, which gives
\eqref{eq:eta-sector}.
\end{proof}

For $x$ in the slit plane $|\arg x|<\pi$, define
\begin{equation}\label{eq:def-f-F}
f_p(x):=h_p(2\sqrt{x}),
\qquad
F_p(x):=\prod_{s\geq0}\bigl(1+f_p(p^{2s}x)\bigr),
\end{equation}
where $\sqrt{x}$ is the principal square root.

\begin{proposition}\label{prop:local-G-F}
In a sufficiently small indented neighborhood of $z=1$,
\begin{equation}\label{eq:local-G-F}
\mathcal G(z)=F_p(\eta(z)).
\end{equation}
\begin{equation}\label{eq:F-functional}
F_p(x)=\bigl(1+f_p(x)\bigr)F_p(p^2x).
\end{equation}
\end{proposition}

\begin{proof}
In a small indented neighborhood of $1$, we have $q=e^{-u}$. Then the definition of $A_s$ together with
Proposition~\ref{prop:one-scale} gives
\[
A_s(z)=1+h_p(p^su(z)).
\]
Because $\eta=u^2/4$, we have
\[
2\sqrt{p^{2s}\eta(z)}=p^su(z)
\]
with the local choices of branches just fixed.  Hence
\[
A_s(z)
=1+h_p(p^su(z))
=1+f_p(p^{2s}\eta(z)).
\]
Taking products over $s\geq0$ proves \eqref{eq:local-G-F}.  Finally,
separating the $s=0$ factor from the product defining $F_p$ gives $\eqref{eq:F-functional}$.
\end{proof}

\section{Mellin analysis and the periodic amplitude}\label{sec:mellin}

In this section, we apply the Mellin transform to $S_p(y)$, a harmonic sum of the logarithm of an infinite product function derived from $\mathcal{G}(z)$. The poles of the transform of the harmonic sum, together with Mellin inversion and the residue theorem, lead to an asymptotic expansion of $S_p(y)$ with a periodic function.

\subsection{The one-scale logarithm}

Put
\begin{equation}\label{eq:def-H-g}
y:=2\sqrt{x},
\qquad
H_p(y):=\prod_{s\geq0}\bigl(1+h_p(p^sy)\bigr),
\qquad
g_p(y):=\log\bigl(1+h_p(y)\bigr).
\end{equation}
Then
\[
F_p(x)=H_p(2\sqrt{x}),
\qquad
\log H_p(y)=\sum_{s\geq0}g_p(p^sy).
\]
We also set
\begin{equation}\label{eq:def-a0}
a_0:=\log(1+h_p(0))=\log\!\left(\frac{p+1}{2}\right).
\end{equation}

By \cite{FlGoDu95}, the Mellin transform of a function $f:(0,\infty)\to\C$ is
defined by
\begin{equation}\label{def:Mellin Transform}
    f^*(s):=\int_0^\infty f(x)x^{s-1}\,dx.
\end{equation}

There is a largest open vertical strip, $\{s=\sigma+ i t \mid \alpha <\sigma <\beta\}$ for some $\alpha,\ \beta \in \R$,  where $f^*(s)$ converges. The number $\alpha$ (resp. $\beta)$ is determined by the order of $f(x)$ when $x \to 0$ (resp. $x \to \infty$). See \cite[Fig. 3]{FlGoDu95}.

\begin{lemma}\label{lem:g-basic}
Fix $0<\theta<\pi/2$ and let
\[
\Sigma_\theta:=\{y\neq0:|\arg y|<\theta\}.
\]
The function $g_p$ has a holomorphic branch on $\Sigma_\theta$ such that
\begin{align}
g_p(y)
 &=a_0-\frac{p(p-1)}{12}y^2+O(y^4)
 &&(y\to0),
 \label{eq:g-small}\\
g_p(y)
 &=O(e^{-\Re y})
 &&(|y|\to\infty),
 \label{eq:g-large}
\end{align}
uniformly on every closed subsector of $\Sigma_\theta$.
\end{lemma}

\begin{proof}
By \eqref{eq:one-plus-hp},
\[
1+h_p(y)=\frac{1-e^{-(p+1)y}}{(1-e^{-y})(1+e^{-py})}.
\]
All zeros and poles of this meromorphic function lie on the imaginary axis.
It is therefore nonzero on $\Sigma_\theta$, which is simply connected, so a
holomorphic logarithm exists there.

For the expansion at zero,
\[
\frac{\sinh(jy)}{\sinh(py)}
 =\frac jp+\frac{j(j^2-p^2)}{6p}y^2+O(y^4).
\]
Summing over $1\leq j\leq p-1$ gives
\[
h_p(y)=\frac{p-1}{2}-\frac{p(p^2-1)}{24}y^2+O(y^4),
\]
and taking the logarithm yields \eqref{eq:g-small}.

Fix $0<\theta'<\theta$. If $y$ lies in the closed subsector
$$|\arg y|\leq \theta',$$ then
$$
\Re y=|y|\cos(\arg y)\geq |y|\cos\theta'.$$
Since $\theta'<\pi/2$, we have $\cos\theta'>0$, and hence $\Re y\to\infty$ uniformly as $|y|\to\infty$ in this subsector.

Recall the exponential expression of $h_p(y)  =({e^{-y}+\cdots+e^{-(p-1)y}})/({1+e^{-py}})$.

For the numerator,
$$\left|e^{-y}+e^{-2y}+\cdots+e^{-(p-1)y}\right|\leq \sum_{j=1}^{p-1} e^{-j\Re y}\leq(p-1)e^{-\Re y}.$$
Moreover, $|e^{-py}|=e^{-p\Re y}\to0$ uniformly in the closed subsector as $|y|\to \infty$. As the denominator $1+e^{-py}$ is bounded below uniformly, it follows that $$ h_p(y)=O\left(e^{-\Re y}\right)$$ as $|y|\to\infty$, uniformly on every closed subsector of $\Sigma_\theta$.
By the expansion of $\log (1+x)$, $g_p(y)=h_p(y)+O(h_p(y)^2)$ and it proves \eqref{eq:g-large}.
\end{proof}

Our next step is to study the Mellin transform of $g_p(y)$. To do that, we consider an auxiliary function
\begin{equation}\label{eq:def-r}
r_p(y):=g_p(y)-a_0(1+y)e^{-y}.
\end{equation}
Since $(1+y)e^{-y}=1-y^2/2+O(y^3)$, Lemma~\ref{lem:g-basic} gives $r_p(y)=O(y^2)$ as $y \to 0$. This regularization ensures that the Mellin transform of $r_p$ is defined in the half-plane $\Re s>-2$.

\begin{proposition}\label{prop:g-mellin}

The Mellin transform $g^*_p(s)$ is holomorphic on $\Re s>0$. Moreover, $g^*_p(s)$ extends meromorphically to $\Re s>-2$ and
\begin{equation}\label{eq:gstar-regularized}
g_p^*(s)=a_0\bigl(\Gamma(s)+\Gamma(s+1)\bigr)+r_p^*(s).
\end{equation}
Here $r_p^*$ is holomorphic on $\Re s>-2$.  The only pole of $g_p^*$ in this
half-plane is a simple pole at $s=0$ with residue $a_0$; in particular, the
apparent pole at $s=-1$ cancels.

Moreover, this transform has exponential decay on every finite vertical strip. To put it precisely, write $s=\sigma+it$. For every $\rho<\pi/2$, on every finite strip
\[
-2<\sigma_1\leq\Re s\leq\sigma_2<\infty
\] we have \begin{equation}\label{eq:gstar-vertical}
g_p^*(\sigma+it)=O(e^{-\rho|t|})
\end{equation}
uniformly for $\sigma\in[\sigma_1,\sigma_2]$ as $|t|\to\infty$. 
\end{proposition}

\begin{proof}
We first consider the Mellin transform of the auxiliary function $r_p(y)$.  Let
$$K\subset\{s\in\mathbb C:\Re s>-2\}$$ be compact, and put
\[
 \sigma_-:=\min_{s\in K}\Re s,
 \qquad
 \sigma_+:=\max_{s\in K}\Re s.
\]
By Lemma~\ref{lem:g-basic} and the definition of $r_p$, we have
$r_p(y)=O(y^2)$ as $y\to0$ and
$r_p(y)=O((1+y)e^{-y})$ as $y\to+\infty$.  Hence, uniformly for $s\in K$,
\[
 |r_p(y)y^{s-1}|\ll y^{\sigma_-+1}
 \quad(0<y\leq1),
\]
and
\[
 |r_p(y)y^{s-1}|
 \ll (1+y)e^{-y}y^{\sigma_+-1}
 \quad(y\geq1).
\]
Both majorants are integrable. Therefore the integral defining $r_p^*(s)$ converges absolutely and locally
uniformly on $\Re s>-2$.  The truncated integrals
\[
 \int_\varepsilon^R r_p(y)y^{s-1}\,\dd y
\]
are entire functions of $s$, and the Weierstrass theorem for locally uniform
limits of holomorphic functions
\cite[Chapter~5, \S1.1, Theorem~1]{Ahlfors79} shows that $r_p^*$ is
holomorphic on $\Re s>-2$.  This is also the standard fundamental-strip
argument for Mellin transforms; see \cite[\S2]{FlGoDu95}.

For $\Re s>0$, the definition of the Gamma function $\Gamma(s)$ gives
\begin{align*}
 \int_0^\infty (1+y)e^{-y}y^{s-1}\,\dd y
 &=\int_0^\infty e^{-y}y^{s-1}\,\dd y
   +\int_0^\infty e^{-y}y^s\,\dd y \\
 &=\Gamma(s)+\Gamma(s+1).
\end{align*}
Since $g_p(y)=a_0(1+y)e^{-y}+r_p(y)$, it follows that
\begin{equation*}
 g_p^*(s)
 =a_0\bigl(\Gamma(s)+\Gamma(s+1)\bigr)+r_p^*(s).
\end{equation*}
This is \eqref{eq:gstar-regularized}.  Its right-hand side is meromorphic on
$\Re s>-2$ and agrees with $g_p^*$ on $\Re s>0$, so it gives the asserted
meromorphic continuation.

The functional equation $\Gamma(s+1)=s\Gamma(s)$ gives
\cite[Eq.~(5.5.1)]{NIST10}
\[
 \Gamma(s)+\Gamma(s+1)=(1+s)\Gamma(s).
\]
In the half-plane $\Re s>-2$, the only poles of $\Gamma(s)$ are the simple
poles at $s=0$ and $s=-1$.  The factor $1+s$
cancels the pole at $s=-1$, whereas it is nonzero at $s=0$.  Since
$\operatorname{Res}_{s=0}\Gamma(s)=1$ and $r_p^*$ is holomorphic there,
$g_p^*$ has a unique pole in this half-plane, namely a simple pole at $s=0$
with residue $a_0$.

It remains to prove the vertical estimate.  Fix $0<\rho<\pi/2$ and choose 
$\rho<\theta<{\pi}/{2}.$

As $r_p(y)=g_p(y)-a_0(1+y)e^{-y}$, Lemma~\ref{lem:g-basic} shows that $r_p$ is holomorphic in
$\Sigma_\theta$, and decays exponentially
at infinity on the closed subsector $|\arg y|\leq\rho$. We also have $r_p(y)=O(y^2)$ at zero. Proposition~5 of
\cite{FlGoDu95} gives
\begin{equation*}\label{eq:rpstar-vertical-proof}
 r_p^*(\sigma+it)=O(e^{-\rho|t|})
\end{equation*}
uniformly for $\sigma$ in any fixed closed finite subinterval of $(-2,\infty)$.

Finally, Stirling's formula in vertical strips 
\cite[Eq.~(5.11.9)]{NIST10} gives 
\[
 |\Gamma(\sigma+it)|
 \ll (1+|t|)^{\sigma-1/2}e^{-\pi|t|/2},
\]
uniformly when $\sigma$ ranges over a fixed finite interval. Hence   $|\Gamma(\sigma+it)|
 \ll e^{-(\pi-\epsilon)|t|/2}$ for any $\epsilon>0$. The same estimate applies to
$\Gamma(\sigma+1+it)$. Hence, $$\Gamma(\sigma+it)+\Gamma(\sigma+1+it)=O(e^{-\rho|t|}).$$

As $g_p^*(s)=a_0\bigl(\Gamma(s)+\Gamma(s+1)\bigr)+r_p^*(s)$, we have \eqref{eq:gstar-vertical} uniformly
for $\sigma\in[\sigma_1,\sigma_2]$.
\end{proof}

\subsection{The pole lattice}

Define the harmonic sum
\begin{equation}\label{eq:def-Sp}
S_p(y):=\sum_{k\geq0}g_p(p^ky)=\log H_p(y).
\end{equation}

For each fixed $y>0$, the series defining $S_p(y)$ converges absolutely,
since Lemma~\ref{lem:g-basic} gives
\[
g_p(p^ky)=O(e^{-p^ky})
\qquad (k\to\infty).
\]
Now let $s=\sigma+it$ with $\sigma>0$.  The same lemma implies
\[
I_\sigma:=\int_0^\infty |g_p(u)|u^{\sigma-1}\,\dd u<\infty:
\]
near zero the integrand is $O(u^{\sigma-1})$, while at infinity it
decays exponentially.  After the change of variables $u=p^ky$, we obtain
\[
\sum_{k\geq0}\int_0^\infty
 |g_p(p^ky)y^{s-1}|\,\dd y
 =
I_\sigma\sum_{k\geq0}p^{-k\sigma}<\infty.
\]
Fubini's theorem therefore applies, and gives
\begin{align}
S_p^*(s)
&=\sum_{k\geq0}\int_0^\infty g_p(p^ky)y^{s-1}\,\dd y \notag\\
&=g_p^*(s)\sum_{k\geq0}p^{-ks}
 =\frac{g_p^*(s)}{1-p^{-s}}.
\label{eq:Sp-Mellin-factor}
\end{align}
This is the standard separation property for harmonic sums; see
\cite[Lemma~2]{FlGoDu95}.

\begin{proposition}\label{prop:Sp-periodic} 
Fix $0<\sigma<1$.  There is a $1$-periodic function
$\widetilde\Phi$, holomorphic in every horizontal strip
\begin{equation}\label{eq:Phi-strip}
|\Im u|<h,
\qquad
h<\frac{\pi}{\log(p^2)},
\end{equation}
such that, uniformly for $y\to0$ in every closed sector
$|\arg y|\leq\theta_0<\pi/2$,
\begin{equation}\label{eq:Sp-asymptotic}
S_p(y)=2\alpha\log\frac1y
 +\widetilde\Phi\!\left(\log_p\frac1y\right)
 +O(|y|^{2\sigma}),
\end{equation}
where
\begin{equation}\label{eq:def-alpha}
\alpha:=\frac{a_0}{2\log p}.
\end{equation}
Consequently, with $\widetilde\Omega:=e^{\widetilde\Phi}$,
\begin{equation}\label{eq:Hp-asymptotic}
H_p(y)=y^{-2\alpha}
\widetilde\Omega\!\left(\log_p\frac1y\right)
\bigl(1+O(|y|^{2\sigma})\bigr).
\end{equation}
\end{proposition}

\begin{proof}
Choose $c>0$.  Mellin inversion (see
\cite[Theorem~2]{FlGoDu95}) and \eqref{eq:Sp-Mellin-factor} give
\begin{equation}\label{eq:Sp-Mellin-inversion}
S_p(y)=\frac1{2\pi i}\int_{c-i\infty}^{c+i\infty}
\frac{g_p^*(s)}{1-p^{-s}}y^{-s}\dd s.
\end{equation}

The decay in \eqref{eq:gstar-vertical} makes the integral absolutely and locally uniformly convergent for $|\arg y|<\pi/2$, and hence it defines a holomorphic function there. For $y>0$, Mellin inversion shows that this function agrees with $S_p(y)$. It therefore gives the analytic continuation of the
series \eqref{eq:def-Sp} throughout $|\arg y|<\pi/2$.

To prove \eqref{eq:Sp-asymptotic}, we consider rectangular contours and shift the contour in \eqref{eq:Sp-Mellin-inversion} to
$\Re s=-2\sigma$.  As 
$|y^{-(-2\sigma+it)}|\leq |y|^{2\sigma}e^{\theta_0|t|}$, while
\eqref{eq:gstar-vertical} is available with an exponent
$\rho>\theta_0$, the integral on the new vertical line is
$O(|y|^{2\sigma})$ uniformly in every closed sector.  

Proposition~\ref{prop:g-mellin} shows that no pole occurs
at $s=-1$.  The poles crossed are $s_m:={2\pi i m}/{\log p}$ for $ m\in\Z.$ To justify the contour shift, truncate the contour by rectangles whose
horizontal sides lie at $\Im s=\pm {(2N+1)\pi}/{\log p}$. On these
horizontal sides, $|1-p^{-s}|$ is bounded away from zero, while
\eqref{eq:gstar-vertical} implies that the corresponding integrals tend
to zero as $N\to\infty$. The residue theorem therefore justifies the
contour shift.

The pole at $s=0$ is double, because both $g_p^*$ and
$(1-p^{-s})^{-1}$ have simple poles there.  Using
\[
g_p^*(s)=\frac{a_0}{s}+O(1),
\qquad
\frac1{1-p^{-s}}=\frac1{s\log p}+\frac12+O(s),
\]
we find that its residue contribution is
\[
\frac{a_0}{\log p}\log\frac1y+C_0
 =2\alpha\log\frac1y+C_0
\]
for a real constant $C_0$.  For $m\neq0$, the pole $s_m$ comes only from the
geometric factor and contributes
\[
\frac{g_p^*(s_m)}{\log p}y^{-s_m}
 =\frac{g_p^*(s_m)}{\log p}
   e^{2\pi i m\log_p(1/y)}.
\]
Let $\widetilde\Phi(u)$ denote the $1$-periodic function defined by
\begin{equation}\label{eq:def-Phi-tilde}
C_0+\frac1{\log p}\sum_{m\neq0}g_p^*(s_m)e^{2\pi imu}.
\end{equation}
Thus the sum of the residues is $2\alpha \log(1/y)+\widetilde\Phi(\log_p(1/y)$ and $$S_p(y)=2\alpha\log\frac1y
 +\widetilde\Phi\!\left(\log_p\frac1y\right)
 +O(|y|^{2\sigma}).$$

It remains to show that the periodic function $\widetilde\Phi$ is
holomorphic in the strips specified in \eqref{eq:Phi-strip}. At $s=s_m={2\pi i m}/{\log p},$
the vertical decay estimate \eqref{eq:gstar-vertical} gives $g_p^*(s_m)=O(e^{-\lambda|m|})$ for every $\lambda<\pi^2/\log p$. If $|\Im u|\leq h$, then $|e^{2\pi imu}|
\leq e^{2\pi h|m|},$
and hence
\[
|g_p^*(s_m)e^{2\pi imu}|
\ll e^{-(\lambda-2\pi h)|m|}.
\]
Therefore, whenever $h<\lambda/(2\pi)$, the Fourier series
\eqref{eq:def-Phi-tilde} converges normally in the strip
$|\Im u|<h$ and therefore defines a holomorphic function there.
Since $\lambda$ may be chosen arbitrarily close to
$\pi^2/\log p$, this holds for every $h<{\pi}/{2\log p}={\pi}/{\log p^2},$
which proves \eqref{eq:Phi-strip}. This completes the proof of
\eqref{eq:Sp-asymptotic}; exponentiating then gives
\eqref{eq:Hp-asymptotic}. 
\end{proof}

\subsection[The asymptotic expansion of $F_p$]{The asymptotic expansion of $F_p$}

\begin{proposition}\label{prop:F-periodic}
Fix $0<\sigma<1$.  There exists a $1$-periodic function $\Omega$, holomorphic
in every strip $|\Im u|<h$ with $h<\pi/\log(p^2)$, such that, uniformly as
$x\to0$ in every closed sector $|\arg x|\leq\pi-\varepsilon$,
\begin{equation}\label{eq:F-periodic-asymptotic}
F_p(x)=x^{-\alpha}
\Omega\!\left(\log_{p^2}\frac1x\right)
\bigl(1+O(|x|^\sigma)\bigr).
\end{equation}
The function $\Omega$ is real and strictly positive on the real axis and is
bounded above and below there by positive constants.
\end{proposition}

\begin{proof}
Set $y=2\sqrt{x}$ in \eqref{eq:Hp-asymptotic}.  Since
\[
(2\sqrt{x})^{-2\alpha}=2^{-2\alpha}x^{-\alpha},
\qquad
\log_p\frac1{2\sqrt{x}}
 =\log_{p^2}\frac1x-\log_p2,
\]
we may take
\begin{equation}\label{eq:def-Omega}
\Omega(u):=2^{-2\alpha}
\widetilde\Omega(u-\log_p2).
\end{equation}
This gives \eqref{eq:F-periodic-asymptotic} and preserves the strip of
holomorphy and period $1$.

For real $u$, the Fourier coefficients in
\eqref{eq:def-Phi-tilde} occur in conjugate pairs, so
$\widetilde\Phi(u)$ is real.  Hence $\widetilde\Omega(u)>0$, and the same is
true of $\Omega(u)$.  A continuous positive periodic function has a positive
minimum and a finite maximum on the real axis.
\end{proof}

\begin{corollary}\label{cor:F-Fourier}
Fix $0<\sigma<1$ and set
\[
\tau_m:=\frac{2\pi m}{\log(p^2)},
\qquad
\beta_m:=\alpha+i\tau_m,
\qquad m\in\mathbb Z.
\]
Then, in every closed sector
\[
|\arg x|\leq\pi-\varepsilon,
\qquad \varepsilon>0,
\]
we have
\begin{equation}\label{eq:F-Fourier-expansion}
F_p(x)
=
\sum_{m\in\mathbb Z}c_mx^{-\beta_m}
+
O\left(|x|^{-\alpha+\sigma}\right)
\qquad (x\to0),
\end{equation}
where $(c_m)_{m\in\mathbb Z}$ are the Fourier coefficients of the
periodic function $\Omega$ in Proposition~\ref{prop:F-periodic}. The series in
\eqref{eq:F-Fourier-expansion} converges normally on compact subsets of the
sector.

Moreover, for every
\begin{equation}\label{eq:cm-decay-threshold}
    0<\lambda_c<\frac{2\pi^2}{\log(p^2)},
\end{equation}
there exists $C_{\lambda_c}>0$ such that
\begin{equation}\label{eq:cm-decay}
|c_m|
\leq
C_{\lambda_c}e^{-\lambda_c|m|}
\qquad(m\in\mathbb Z).
\end{equation}
In particular, $\lambda_c$ may be chosen so that $\lambda_c>{\pi^2}/{\log(p^2)}.$
\end{corollary}

\begin{proof}
For every $h<\pi/\log(p^2)$, since $\Omega(u)$ is periodic and holomorphic in the strip
$|\Im u|<h$, its Fourier coefficients $c_m$ satisfy
\[
|c_m|\leq C_he^{-2\pi h|m|}.
\]
This is equivalent to \eqref{eq:cm-decay} for every $\lambda_c$ in
\eqref{eq:cm-decay-threshold}.  For a fixed sector
$|\arg x|\leq\pi-\varepsilon$, choose $h>(\pi-\varepsilon)/\log(p^2)$.
The Fourier series then remains normally convergent after the substitution
$u=\log_{p^2}(1/x)$.  Using \eqref{eq:F-periodic-asymptotic} gives
\eqref{eq:F-Fourier-expansion}.
\end{proof}

\section{Boundary singularities and coefficient extraction}\label{sec:transfer}

In this section we study the boundary singularities of $\mathcal B(z)$ in order
to extract the asymptotics of its coefficients by singularity analysis
\cite[Chapter~VI]{FlSe09}. In its basic $O$-transfer form, the method
says, roughly, that if $f$ is analytic in a suitable indented domain at $z=1$ and
\[
f(z)=O\left(|1-z|^{-\lambda}\right)
\qquad (z\to1)
\]
for some $\lambda>0$, then
\[
[z^k]f(z)=O\left(k^{\lambda-1}\right).
\]
When $f$ has finitely many singularities $z_i$ such that  $f(z)=O(| z-z_i|^{-\lambda_i})$ as $z \to z_i$  with a suitable analyticity property in their neighborhoods, then $[z^k]f(z)= \sum_{i} O\left(k^{\lambda_i-1}\right)$. Thus the coefficient problem is reduced to determining the local behavior
of the generating function at its boundary singularities.

As $\mathcal B(z)={2}(\mathcal G(z)-1)/z,$
Proposition~\ref{prop:global-continuation} shows that the only possible
singularities of $\mathcal B$ on the unit circle are $z=1$ and $z=-1$. The point
$z=1$ is already known to be a singularity. We will show below
that $z=-1$ is also a boundary singularity of fractional order. 

Near $z=1$, Corollary~\ref{cor:F-Fourier} gives a periodic expansion of $F_p$. As
$\mathcal G(z)=F_p(\eta(z))$ by Proposition \ref{prop:local-G-F} and $\eta(z) \sim (1-z)/2$ by Lemma \ref{lem:eta-local}, we define a singular model
$\mathcal S(z)$ by substitute $x=(1-z)/2$ into be the periodic part of the expansion of $F_p(x)$. Applying the transfer method to the individual terms of $\mathcal S$ suggests the form of the coefficient asymptotic and motivates the
definition of the function $\omega$. We then show that $\mathcal B-\mathcal S$
is of strictly smaller order than $\mathcal{S}$ near $z=1$. Thus, $\mathcal S$ produces the main contribution arising from
$z=1$. 

We next study the second boundary point $z=-1$. Using the local variable
$\mu(z)$, we obtain a weaker singular estimate there. Its contribution
to the coefficients is of lower order. 

Finally, we use the multiple-point transfer argument to obtain
the asymptotic formula for the coefficients of $\mathcal B$.

\subsection[The first boundary point $z=1$]{The first boundary point $z=1$}

Let $(c_m)$ be the Fourier coefficients from Corollary~\ref{cor:F-Fourier} and
set
\begin{equation}\label{eq:def-dm}
d_m:=2^{1+\beta_m}c_m,
\qquad
\beta_m=\alpha+i\tau_m.
\end{equation}
The coefficients $d_m$ satisfy
\begin{equation}\label{eq:dm-decay}
|d_m|\leq Ce^{-\lambda_d|m|}
\end{equation}
for some
\begin{equation}\label{eq:lambda-d}
\lambda_d>\frac{\pi^2}{\log(p^2)}.
\end{equation}
Choose the indentation angle $\phi$ sufficiently close to $\pi/2$ that
\begin{equation}\label{eq:angle-condition}
\frac{2\pi(\pi-\phi)}{\log(p^2)}<\lambda_d.
\end{equation}
All powers of $1-z$ below use the branch of $\log(1-z)$ obtained by
continuation from $0<z<1$ inside $\Delta_{\phi,\rho}$.

Define the pure periodic singular model \begin{equation}\label{eq:def-singular-model}
\mathcal S(z):=\sum_{m\in\Z}d_m(1-z)^{-\beta_m}.
\end{equation}
Condition \eqref{eq:angle-condition} implies normal convergence on compact
subsets of $\Delta_{\phi,\rho}$. It is clear that $\mathcal{S}(z)$ is singular at $z=1$ while analytic across $z=-1$ from \eqref{eq:def-singular-model} and \eqref{eq:dm-decay}.

The next proposition shows that $\mathcal{B}(z)$ is well approximated by the singular model $\mathcal{S}$ when $z$ approaches $1$ in $\Delta_{\phi,\rho}$.

\begin{proposition}\label{prop:B-minus-S-at-one}
Fix $0<\sigma<1$. As $z\to1$ in $\Delta_{\phi,\rho}$,
\begin{equation}\label{eq:B-minus-S-one}
\mathcal B(z)-\mathcal S(z)
 =-2+O\bigl(|1-z|^{-\alpha+\sigma}\bigr).
\end{equation}
\end{proposition}

\begin{proof}

Recall that $1+(z/2)\mathcal{B}(z)=F_p(\eta(z))$ near $z=1$ by Lemma \ref{lem:B-via-G} and Proposition \ref{prop:local-G-F} where $$F_p(x)=\sum_{m\in\Z}c_mx^{-\beta_m}
       +O(|x|^{-\alpha+\sigma})
\qquad(x\to0)$$ by Corollary \ref{cor:F-Fourier} and $\eta(z)$ defined \eqref{eq:def-u-eta-local}. Therefore, \begin{equation}\label{eq: B(z) c_m and eta}
    \mathcal{B}(z)=2/z \Big(\sum_{m \in \Z} c_m \eta(z)^{-\beta_m} +O(| \eta(z)|^{-\alpha+\sigma}) -1 \Big).
\end{equation}

Write $\eta(z)=(1-z)\Xi(z)$ as in Lemma \ref{eq:eta-factor} and
$d_m=2^{1+\beta_m}c_m$, we have 
\begin{equation}\label{eq:am-definition}
 \frac{2}{z}c_m\eta(z)^{-\beta_m}
 =d_m a_m(z)(1-z)^{-\beta_m},
 \qquad
 a_m(z):=z^{-1}(2\Xi(z))^{-\beta_m}.
\end{equation}
The branches in this formula are the ones vanishing at $1$.  In particular, $a_m(1)=1$.

First, we show that the series $\sum_{m \in \Z} d_m (a_m(z)-1)(1-z)^{- \beta_m}$ converges absolutely and is $O(|1-z|^{1-\alpha}).$ 

We start with estimating $a_m(z)-1$ uniformly in $m$.  Set $w_m(z):=\log a_m(z) =-\log z-\beta_mL(z),$ where $L(z)$ is defined as $\log(2\Xi(z))$. After shrinking the neighborhood of $1$, we may assume both $\log z$ and $L(z)$ are analytic near $z=1$. Analyticity then gives $ |L(z)|\ll |1-z|$ and $|\log z|\ll |1-z|.$ Since $\beta_m=\alpha+{2\pi i m}/{\log(p^2)}$, the triangle inequality gives
\[
|\beta_m|\leq \alpha+\frac{2\pi}{\log(p^2)}|m|\ll 1+|m|,
\]
uniformly in $m$.
These estimates imply
\begin{equation*}\label{eq:wm-bound}
 |w_m(z)|\ll (1+|m|)|1-z|.
\end{equation*}
Applying the elementary inequality
$ |e^w-1|\leq |w|e^{|w|}$ for $w \in \C$
with $w=w_m(z)$  gives
\begin{equation*}\label{eq:am-bound}
 |a_m(z)-1|
 \ll (1+|m|)|1-z|e^{C|m||1-z|}.
\end{equation*}
for some constant $C$.

On the indented domain,
$|\arg(1-z)|\leq\pi-\phi$, and hence
$ |(1-z)^{-i\tau_m}|\leq e^{(|\tau_m|(\pi-\phi))}.$ Therefore, $ |(1-z)^{-i\tau_m}| \leq e^{K|m|}$ for a positive constant $K$ and \begin{equation}\label{eq:power-bound}
 |(1-z)^{-\beta_m}|
 \leq |1-z|^{-\alpha}e^{K|m|}.
\end{equation}
Choose $r_0>0$ sufficiently small that $Cr_0<\lambda_d-K.$

For $|1-z|<r_0$, the bounds on $d_m$, \eqref{eq:am-bound}, and
\eqref{eq:power-bound} yield
\begin{align*}
 &\sum_{m\in\mathbb Z}
 |d_m|\,|a_m(z)-1|\,|(1-z)^{-\beta_m}| \\
 &\qquad\ll
 |1-z|^{1-\alpha}
 \sum_{m\in\mathbb Z}(1+|m|)
 e^{-(\lambda_d-K-Cr_0)|m|} \\
 &\qquad=O\bigl(|1-z|^{1-\alpha}\bigr).
\end{align*}

As the series $\sum_{m\in\mathbb Z}
 |d_m|\,|a_m(z)-1|\,|(1-z)^{-\beta_m}|$ and $S(z)$ converges near $z=1$, we may  have term-by-term operations on the series and have
\begin{equation}\label{eq:main-series-comparison}
 \frac{2}{z}\sum_{m\in\mathbb Z}c_m\eta(z)^{-\beta_m}
 -\Scal(z)
 =O\bigl(|1-z|^{1-\alpha}\bigr).
\end{equation}

Combining \eqref{eq: B(z) c_m and eta}, \eqref{eq:am-definition} and \eqref{eq:main-series-comparison} and using $|\eta(z)|\asymp |1-z|$ and $-{2}/{z}=-2+O(|1-z|),$ we obtain $$
\mathcal B(z)-\mathcal S(z)
=
-2
+
O\left(|1-z|^{1-\alpha}\right)
+
O\left(|1-z|^{-\alpha+\sigma}\right)
+
O(|1-z|).
$$
Since $\sigma<1$, it follows that $$
\mathcal B(z)-\mathcal S(z)
=
-2+O\left(|1-z|^{-\alpha+\sigma}\right).
$$
\end{proof}

We now pass from the singular model $\mathcal{S}$ to its corresponding coefficient asymptotics.
For a fixed $m$, by \cite[Theorem~VI.1]{FlSe09}
\[
[z^k](1-z)^{-\beta_m}
\sim
\frac{k^{\beta_m-1}}{\Gamma(\beta_m)}
=
k^{\alpha-1}
\frac{k^{i\tau_m}}{\Gamma(\beta_m)}.
\]
Thus, after summing over the Fourier modes, the natural candidate for
the periodic amplitude in the coefficient asymptotics is

\begin{equation}\label{eq:def-omega}
    \omega(t):=
\sum_{m\in\mathbb Z}
\frac{d_m}{\Gamma(\beta_m)}t^{i\tau_m}.
\end{equation}

\begin{lemma}\label{lem:periodic-transfer}
Let $\omega$ be given by \eqref{eq:def-omega}. Then its defining series
converges absolutely for every $t>0$. Moreover, the function
$$
u\longmapsto\omega(e^u),
\qquad u\in\R,
$$
extends holomorphically to a horizontal strip containing $\R$, and this
extension is $\log(p^2)$-periodic. In particular, $\omega$ is
smooth in the logarithmic variable and satisfies
\begin{equation}\label{eq:omega-periodic}
\omega(p^2t)=\omega(t),
\qquad t>0.
\end{equation}
Furthermore,
\begin{equation}\label{eq:S-coefficients}
[z^k]\mathcal S(z)
=
k^{\alpha-1}\omega(k)+O(k^{\alpha-2}).
\end{equation}
\end{lemma}

\begin{proof}
We first establish exponential decay of the coefficients in
\eqref{eq:def-omega}. Set
$$
\lambda_\Gamma:=\frac{\pi^2}{\log(p^2)}.
$$
Stirling's formula on the vertical line $\Re s=\alpha$
\cite[Eq.~(5.11.9)]{NIST10}
gives 
\begin{equation}\label{eq:inverse-gamma-growth}
\frac{1}{|\Gamma(\beta_m)|}
\ll
(1+|m|)^A e^{\lambda_\Gamma|m|}
\end{equation}
for some $A>0$. Since
$|d_m|\ll e^{-\lambda_d|m|}$ and $\lambda_d>\lambda_\Gamma,$
we may choose
$
0<\lambda_\omega<\lambda_d-\lambda_\Gamma
$ and obtain
\begin{equation}\label{eq:omega-coeff-decay}
\left|\frac{d_m}{\Gamma(\beta_m)}\right|
\ll e^{-\lambda_\omega|m|}.
\end{equation}
Since $|t^{i\tau_m}|=1$ for $t>0$, this proves the absolute convergence
of \eqref{eq:def-omega}.
\\

Define
\begin{equation}\label{eq:omega-complex-u}
W(u):=
\sum_{m\in\Z}
\frac{d_m}{\Gamma(\beta_m)}e^{i\tau_m u},
\qquad u\in\C.
\end{equation}
Put
$$
h_0:=\frac{\lambda_\omega\log(p^2)}{2\pi}.
$$
For every $0<h<h_0$ and $|\Im u|\le h$, we have
$$
\left|
\frac{d_m}{\Gamma(\beta_m)}e^{i\tau_m u}
\right|
\ll
\exp\left(
-\left(
\lambda_\omega-\frac{2\pi h}{\log(p^2)}
\right)|m|
\right)
$$
so \eqref{eq:omega-complex-u} converges normally and holomorphic in $|\Im u|<h_0$. Therefore,
the function $u\mapsto\omega(e^u)$ is smooth on $\R$. As $e^{i\tau_m\log(p^2)}=e^{2\pi im}=1,$
 $W$ is $\log(p^2)$-periodic, and
$
\omega(p^2t)
=\omega(t),
$
which proves \eqref{eq:omega-periodic}.
\\

We now prove \eqref{eq:S-coefficients}. The idea is that we write both $[z^k]\mathcal{S}(z)$ and $k^{\alpha-1} \omega(k)$ as series in $m$, and truncate them at $|m|\le M$ where $M$ will be chosen. We then show that the truncated sums satisfy \eqref{eq:S-coefficients}, and the tails are negligible. 

For every $m$, the generalized binomial theorem gives
\begin{equation}\label{eq:exact-power-coefficient}
[z^k](1-z)^{-\beta_m}
=
\frac{\Gamma(k+\beta_m)}
{\Gamma(\beta_m)\Gamma(k+1)}.
\end{equation}
Hence, $$[z^k] \mathcal{S}(z)=\sum_{m \in \Z} \frac{d_m \Gamma(k+\beta_m)}{\Gamma(\beta_m) \Gamma(k+1)}.$$ By the definition of $\omega(t)$, $$k^{\alpha-1} \omega(k)=\sum_{m \in \Z} \frac{d_m}{\Gamma(\beta_m)} k^{\beta_m-1}.$$

Set $\lambda_\phi
:=
{2\pi(\pi-\phi)}/{\log(p^2)}$. Then $\lambda_\phi<\lambda_d$ by \eqref{eq:angle-condition}.
Choose $A_1>0$ so that
$$
A_1(\lambda_d-\lambda_\phi)\ge1,
\qquad
A_1\lambda_\omega\ge1,
$$
and put
$$
M:=\lfloor A_1\log k\rfloor.
$$
\\

For $|m|\le M$, we have $|\beta_m|\ll1+|m|\ll\log k.$
We claim that
\begin{equation}\label{eq:uniform-gamma-ratio}
\frac{\Gamma(k+\beta_m)}
{\Gamma(k+1)k^{\beta_m-1}}
=
1+
O\left(
\frac{(1+|\beta_m|)^2}{k}
\right),
\end{equation}
with an implied constant independent of $m$ in this range.

Let $\psi(z):={\Gamma'(z)}/{\Gamma(z)}.$ The digamma asymptotic
\cite[Eq.~(5.11.2)]{NIST10} gives
$$
\psi(z)=\log z+O\left(\frac1z\right)
$$
in sectors bounded away from the negative real axis. If $w$ lies on
the segment from $1$ to $\beta_m$, then $|w|\ll1+|\beta_m|\ll\log k.$ Consequently, $|k+w|\asymp k$, and
$$
\psi(k+w)-\log k
=
O\left(\frac{1+|w|}{k}\right).
$$
Integrating along this segment yields
$$
\begin{aligned}
&\log\Gamma(k+\beta_m)-\log\Gamma(k+1)
-(\beta_m-1)\log k \\
&\qquad=
\int_1^{\beta_m}
\bigl(\psi(k+w)-\log k\bigr)\,dw \\
&\qquad=
O\left(\frac{(1+|\beta_m|)^2}{k}\right).
\end{aligned}
$$
Since $(1+|\beta_m|)^2/k=O((\log k)^2/k)=o(1)$, exponentiation proves
\eqref{eq:uniform-gamma-ratio}.

Combining \eqref{eq:exact-power-coefficient} and
\eqref{eq:uniform-gamma-ratio}, we obtain
$$
\begin{aligned}
[z^k]\sum_{|m|\le M}d_m(1-z)^{-\beta_m}
&=
k^{\alpha-1}
\sum_{|m|\le M}
\frac{d_m}{\Gamma(\beta_m)}k^{i\tau_m} \\
&\quad+
O\left(
k^{\alpha-2}
\sum_{|m|\le M}
(1+|\beta_m|)^2
\left|\frac{d_m}{\Gamma(\beta_m)}\right|
\right).
\end{aligned}
$$
By \eqref{eq:omega-coeff-decay} and
$|\beta_m|\ll1+|m|$, the sum in the error term is bounded independently
of $M$. Hence
\begin{equation}\label{eq:central-mode-estimate}
[z^k]\sum_{|m|\le M}d_m(1-z)^{-\beta_m}
=
k^{\alpha-1}
\sum_{|m|\le M}
\frac{d_m}{\Gamma(\beta_m)}k^{i\tau_m}
+
O(k^{\alpha-2}).
\end{equation}

It remains to show that the tails are $O(k^{\alpha-2})$. On the fixed $\Delta$-domain, we have 
$
|\arg(1-z)|\le\pi-\phi,
$
and therefore
$$
|(1-z)^{-\beta_m}|
\le
|1-z|^{-\alpha}e^{\lambda_\phi|m|}.
$$
Together with \eqref{eq:dm-decay}, this gives
$$
|d_m(1-z)^{-\beta_m}|
\ll
|1-z|^{-\alpha}
e^{-(\lambda_d-\lambda_\phi)|m|}.
$$
Thus
\begin{equation}\label{eq:S-tail-bound}
\left|
\sum_{|m|>M}d_m(1-z)^{-\beta_m}
\right|
\ll
e^{-(\lambda_d-\lambda_\phi)M}|1-z|^{-\alpha},
\end{equation}
with an implied constant independent of $M$. The same estimate gives
normal convergence of the tail in the fixed $\Delta$-domain. Hence the
$O$-transfer theorem
\cite[Theorem~VI.3]{FlSe09} gives
$$
[z^k]\sum_{|m|>M}d_m(1-z)^{-\beta_m}
=
O\left(
e^{-(\lambda_d-\lambda_\phi)M}k^{\alpha-1}
\right),
$$
again with an implied constant independent of $M$. 

Similarly, \eqref{eq:omega-coeff-decay} gives
$$
\left|
\sum_{|m|>M}
\frac{d_m}{\Gamma(\beta_m)}k^{i\tau_m}
\right|
\ll e^{-\lambda_\omega M}.
$$
Since $M=\lfloor A_1\log k\rfloor$, the choice of $A_1$ implies
$$
e^{-(\lambda_d-\lambda_\phi)M}\ll k^{-1},
\qquad
e^{-\lambda_\omega M}\ll k^{-1}.
$$
Therefore, the tails of $[z^k]\mathcal S(z)$ and $k^{\alpha-1}\omega(k)$ are of order $O(k^{\alpha-2})$. Combining these estimates with the result for the truncated sums, we have $$
[z^k]\mathcal S(z)
=
k^{\alpha-1}\omega(k)+O(k^{\alpha-2}).
$$
\end{proof}

\subsection[The secondary boundary point $z=-1$]{The secondary boundary point $z=-1$}

The Joukowski coordinate also gives a clean local model at the second boundary
point.  Near $z=-1$, define
\begin{equation}\label{eq:def-mu}
\mu(z):=\arcosh\!\left(-\frac1z\right),
\end{equation}
using the branch that tends to $0$ as $z\to-1$ inside the indentation.

\begin{lemma}\label{lem:mu-local} 
Near $z=-1$,
\begin{equation}\label{eq:q-mu}
q(z)=-e^{-\mu(z)}.
\end{equation}
After shrinking the indented neighborhood of $-1$, there is an analytic function $\Theta_-$, with $$\Theta_-(-1)=1,$$ such that
\begin{equation}\label{eq:mu-factor}
\mu(z)=\sqrt{2(1+z)}\,\Theta_-(z).
\end{equation}
Consequently,
\begin{equation}\label{eq:mu-size}
|\mu(z)|\asymp|1+z|^{1/2},
\end{equation}
and $\mu(z)$ remains in a sector $|\arg\mu|<\theta_-$ for some
$\theta_-<\pi/2$.
\end{lemma}

\begin{proof}
By \eqref{eq:Joukowski-inverse},
\[
-\frac1z=-\frac{q+q^{-1}}2.
\]
If $q=-e^{-\mu}$, the right-hand side is $\cosh\mu$.  The branch tending to
zero is therefore exactly \eqref{eq:def-mu}, which proves
\eqref{eq:q-mu}.  Since
\[
-\frac1z-1=(1+z)+O((1+z)^2),
\]
the local expansion $\arcosh(1+w)=\sqrt{2w}(1+O(w))$ gives
\eqref{eq:mu-factor}, and \eqref{eq:mu-size} follows immediately.  The sector assertion follows as $z$ is in the indented neighborhood of $-1$.
\end{proof}

Set
\begin{equation}\label{eq:def-Mp}
M_p(v):=
\frac{1-e^{-(p+1)v}}{(1+e^{-v})(1-e^{-pv})}.
\end{equation}
The singularity at $v=0$ is removable and $M_p(v)={\Lambda}/{p}+O(v),$ 
where $$\Lambda:={(p+1)}/{2}=1+h_p(0).$$

By Lemma~\ref{lem:mu-local}, $q=-e^{-\mu}$ near $-1$.  Since $p$ is odd,
$q^{p^s}=-e^{-p^s\mu}$.  Substituting this into
\eqref{eq:As-product-form} gives
\[
A_s(z)=M_p(p^s\mu(z))
\]
and 
\begin{equation} \label{eq: Gz_and_Mpmiu}
    \mathcal{G}(z)=\prod_{s \ge 0} A_s(z)=\prod_{s \ge 0} M_p(p^s\mu(z)).
\end{equation}
Let
$$
J_p(v):=\prod_{s\geq0}M_p(p^sv).$$

\begin{lemma}\label{lem:Jp-bound}
Let $\gamma:=\log_p({p}/{\Lambda})=1-2\alpha$. Then, uniformly as $v\to0$ in every closed sector
$|\arg v|\leq\theta<\pi/2$,
\begin{equation}\label{eq:Jp-bound}
J_p(v)=O(|v|^\gamma).
\end{equation}
Moreover, along the positive real axis there are constants $c,C>0$ such that
\begin{equation}\label{eq:Jp-two-sided-real}
c v^\gamma\leq J_p(v)\leq C v^\gamma
\qquad(0<v\ll1).
\end{equation}
\end{lemma}

\begin{proof}
Choose $v_0>0$ so that in the chosen sector and for $|w|\leq v_0$,
\[
M_p(w)=\frac{\Lambda}{p}(1+E(w)),
\qquad E(w)=O(w).
\]
For small $v$, choose $N$ so that
$v_0/p< p^N|v| \le v_0$. Splitting at scale $N$ gives
\[
J_p(v)=
\left(\frac{\Lambda}{p}\right)^N
\prod_{s=0}^{N-1}(1+E(p^sv))\,J_p(p^Nv).
\]
The middle product is bounded because
$\sum_{s<N}|E(p^sv)|\ll\sum_{s<N}p^s|v|\ll1$. Since $J_p(\mu(z))=\mathcal{G}(z)$ where $\mu(z)=\arcosh(-1/z)$, and $\cosh v>1$ for any $v>0$,  $J_p(v)$ is continuous on $v>0$ by Proposition \ref{prop:global-continuation}. Therefore, the final product is uniformly
bounded because $p^Nv$ stays in a fixed compact annular sector. Since $\Lambda/p=p^{-\gamma}$ and
$p^N|v|\asymp1$, it follows that $(\Lambda/p)^N \asymp |v|^\gamma.$ Hence, we have  \eqref{eq:Jp-bound}.

To prove \eqref{eq:Jp-two-sided-real}, the upper bound already follows from
\eqref{eq:Jp-bound}, so it remains to find a positive lower bound.

For positive real $v$, all factors $M_p(p^s v)$ are positive. After decreasing
$v_0$ if necessary, the product $\prod_{s<N}(1+E(p^s v))$ is bounded below by a positive constant.

As $p^N v$ lies in the compact interval $[v_0/p,v_0]$. The function $J_p$
is continuous and strictly positive on this interval, so it has a positive
minimum there. Hence $J_p(p^N v)$ is bounded below by a positive constant. Combining these lower bounds with $(\Lambda/p)^N \asymp |v|^\gamma$ gives $J_p(v)\gg v^\gamma.$ 

Together with the upper bound from \eqref{eq:Jp-bound}, this proves
\eqref{eq:Jp-two-sided-real}.

\end{proof}

\begin{proposition}\label{prop:minus-one-bound}

As $z\to-1$ inside $\Delta_{\phi,\rho}$,
\begin{equation}\label{eq:B-at-minus-one}
\mathcal B(z)-2=O(|1+z|^{1/2-\alpha}).
\end{equation}
Along the real interval $z\downarrow-1$ from the right,
\begin{equation}\label{eq:B-minus-one-two-sided}
|\mathcal B(z)-2|\asymp(1+z)^{1/2-\alpha}.
\end{equation}
In particular, $z=-1$ is a boundary singularity.  Consequently,
\begin{equation}\label{eq:B-minus-S-minus-one}
\mathcal B(z)-\mathcal S(z)
 =2-\mathcal S(-1)+O(|1+z|^{1/2-\alpha}).
\end{equation}
\end{proposition}

\begin{proof}

The equation \eqref{eq: Gz_and_Mpmiu} and Lemma~\ref{lem:Jp-bound}
give
\[
\mathcal G(z)=J_p(\mu(z))
 =O(|\mu(z)|^{1-2\alpha})
 =O(|1+z|^{1/2-\alpha}).
\]
First, observe that as $0<\frac{1}{2}-\alpha<1$, we may replace $O(|1+z|)$ by $O(|1+z|^{1/2-\alpha})$ near $z=-1.$ Using \eqref{eq:B-via-G},
\[
\mathcal B(z)-2
 =\frac2z\mathcal G(z)-\frac{2(1+z)}z,
\]
which proves \eqref{eq:B-at-minus-one}.  For real $-1<z<0$, the variable
$\mu(z)$ is positive and \eqref{eq:B-minus-one-two-sided} follows immediately from 
Lemma~\ref{lem:Jp-bound}. The equation $\eqref{eq:B-minus-one-two-sided}$ shows that analytic continuation across
$-1$ is impossible. Finally, $\eqref{eq:B-minus-S-minus-one}$ follows from that $S(z)$ is analytic at $-1$ and replacing $O(|1+z|)$ by $O(|1+z|^{1/2-\alpha})$.
\end{proof}

\subsection{Two-point transfer and the asymptotic formula}

We use the following standard two-point form of singularity analysis.  It is
the usual composite-contour argument applied to the two indentations; see
\cite[Chapter~VI]{FlSe09} and \cite{FlOd90}.

\begin{lemma}[Two-point $O$-transfer]\label{lem:two-point-transfer}
Let $H$ be analytic in $\Delta_{\phi,\rho}$ and suppose that
\[
H(z)=O(|1-z|^{-\lambda_+})
\quad(z\to1),
\qquad
H(z)=O(|1+z|^{-\lambda_-})
\quad(z\to-1),
\]
where $\lambda_+,\lambda_-\in\R$.  Then
\begin{equation}\label{eq:two-point-transfer}
[z^k]H(z)
 =O(k^{\lambda_+-1})+O(k^{\lambda_--1}).
\end{equation}
\end{lemma}

\begin{proof} 
See \cite[Theorem~VI.6]{FlSe09}.
\end{proof}

The values of difference $\mathcal B-\mathcal S$ at $z=1$ and $z=1$ are finite by Propositions \ref{prop:B-minus-S-at-one} and \ref{prop:minus-one-bound}.  A raw $O$-transfer estimate would make the estimate not sharp enough.  We remove
them with a linear interpolation polynomial, which has no effect on
coefficients of degree at least $2$.

\begin{proposition}\label{prop:additive-asymptotic}
We have 
\begin{equation}\label{eq:additive-asymptotic}
2^{-k}b_k
 =k^{\alpha-1}\omega(k)
  +O(k^{\alpha-3/2}).
\end{equation}
Equivalently, with $\delta=1-\alpha$,
\begin{equation}\label{eq:additive-delta}
2^{-k}b_k
 =k^{-\delta}\omega(k)
  +O(k^{-\delta-1/2}).
\end{equation}
\end{proposition}

\begin{proof}
Define $P(z)=-2\frac{1+z}{2}
 +(2-\mathcal S(-1))\frac{1-z}{2}$ so that $P(z)=\mathcal B(z)-\mathcal S(z)$ at $z=1$ and $z=-1$.
Put $R(z):=\mathcal B(z)-\mathcal S(z)-P(z).$ For any $\sigma<1$,  Propositions~\ref{prop:B-minus-S-at-one} and \ref{prop:minus-one-bound} give
\[
R(z)=O(|1-z|^{\sigma -\alpha})
\quad(z\to1),
\qquad
R(z)=O(|1+z|^{1/2-\alpha})
\quad(z\to-1).
\]
By Lemma~\ref{lem:two-point-transfer},
\[[z^k]R(z)
 =O(k^{\alpha-\sigma-1})+O(k^{\alpha-3/2}).
\]
Taking $\sigma=1/2$, then the proposition follows.

\end{proof}

\begin{remark}\label{rem:half-barrier}
In the two-point transfer argument, the
exponent, $\alpha-3/2$, of the error term comes from the singularity at $z=-1$.
\end{remark}

\section{A positive theta kernel from boundary heat flux}\label{sec:theta}

The Fourier expansion in Lemma~\ref{lem:periodic-transfer} defines the
coefficient amplitude, but it does not make its sign transparent.  We now
construct a positive real-variable model.  The key observation is that the
one-scale factor $f_p$ is the Laplace transform of a boundary heat-flux kernel.

\subsection{Definition and basic properties}

Let $K_p(x,y,s)$ be the Dirichlet heat kernel on the interval $(0,p)$:
\begin{equation}\label{eq:heat-kernel-spectral}
K_p(x,y,s)=\frac2p\sum_{n\geq1}
 e^{-\pi^2n^2s/p^2}
 \sin\!\left(\frac{\pi nx}{p}\right)
 \sin\!\left(\frac{\pi ny}{p}\right).
\end{equation}
For $t>0$, define 
\begin{equation}\label{eq:def-varphi}
\varphi_p(t):=\frac14\sum_{a=1}^{p-1}
(\partial_yK_p)(a,0,t/4),
\end{equation}
and set $\varphi_p(t)=0$ for $t\leq0$.

The normalisations are chosen so that the relation between $f_p$ and $\phi_p$ is cleaner (see Proposition~\ref{prop:laplace-varphi}).
\begin{lemma}\label{lem:finite-sine-sum}
For every integer $n$,
\begin{equation}\label{eq:finite-sine-sum}
\sum_{a=1}^{p-1}\sin\!\left(\frac{\pi an}{p}\right)
 =
\begin{cases}
\displaystyle \cot\!\left(\frac{\pi n}{2p}\right),
 & n\text{ odd and }p\nmid n,\\[2mm]
0,&\text{otherwise }.
\end{cases}
\end{equation}
\end{lemma}

\begin{proof}
For $x\notin2\pi\Z$,
\[
\sum_{a=1}^{p-1}\sin(ax)
 =\frac{\sin((p-1)x/2)\sin(px/2)}{\sin(x/2)}.
\]
Substitute $x=\pi n/p$.  The sum vanishes when $n$ is even or $p\mid n$.
If $n$ is odd and $p\nmid n$, then
\[
\sin\!\left(\frac{(p-1)\pi n}{2p}\right)
 =(-1)^{(n-1)/2}\cos\!\left(\frac{\pi n}{2p}\right),
\]
while $\sin(\pi n/2)=(-1)^{(n-1)/2}$.  Their product divided by
$\sin(\pi n/(2p))$ gives the cotangent in \eqref{eq:finite-sine-sum}.
\end{proof}

\begin{proposition}\label{prop:varphi-properties}
The function $\varphi_p$ is strictly positive on $(0,\infty)$, and its
extension by zero is a Schwartz function on $\R$ supported in $[0,\infty)$.
It has the spectral expansions
\begin{align}
\varphi_p(t)
 &=\frac{\pi}{2p^2}\sum_{a=1}^{p-1}\sum_{n\geq1}
 n\sin\!\left(\frac{\pi an}{p}\right)
 e^{-\pi^2n^2t/(4p^2)},
 \label{eq:varphi-spectral}\\
 &=\frac{\pi}{2p^2}
 \sum_{\substack{n\geq1\\ n\text{ odd},\ p\nmid n}}
 n\cot\!\left(\frac{\pi n}{2p}\right)
 e^{-\pi^2n^2t/(4p^2)}.
 \label{eq:varphi-cot}
\end{align}
It also has the image expansion
\begin{equation}\label{eq:varphi-image}
\varphi_p(t)=\frac1{\sqrt\pi\,t^{3/2}}
\sum_{a=1}^{p-1}\sum_{m\in\Z}
(a+2mp)e^{-(a+2mp)^2/t}.
\end{equation}
\end{proposition}

\begin{proof}
Differentiating \eqref{eq:heat-kernel-spectral} at $y=0$ and putting $s=t/4$
gives \eqref{eq:varphi-spectral}.  Lemma~\ref{lem:finite-sine-sum} then gives
\eqref{eq:varphi-cot}.

For positivity, fix $a\in{1,\dots,p-1}$ and $s>0$.  As a function of $(y,\tau)$, the kernel $K_p(a,y,\tau)$ is a positive solution of the heat equation in $(0,p)\times(0,\infty)$ and vanishes on the lateral boundary $y=0$.  The parabolic Hopf boundary point lemma (see e.g.  \cite[Lemma 2.6]{Lieberman96}) , applied at $(0,s)$, therefore gives

$$(\partial_yK_p)(a,0,s)>0,$$
since the positive $y$-direction is the inward normal direction at the left endpoint.  Thus every summand in \eqref{eq:def-varphi} is strictly positive.

The method-of-images formula obtained via Poisson summation of the standard expansion of $K_p(x,y,s)$ gives us
\[
K_p(x,y,s)=\frac1{\sqrt{4\pi s}}
\sum_{m\in\Z}
\left(
 e^{-(x-y+2mp)^2/(4s)}
 -e^{-(x+y+2mp)^2/(4s)}
\right).
\]
Differentiate at $y=0$, set $x=a$ and $s=t/4$, and sum over $a$.  This yields
\eqref{eq:varphi-image}.

The spectral expansion shows, after termwise differentiation, that $\varphi_p$ and all of its derivatives decay exponentially as $t\to\infty$.

For $t\downarrow0$, use the image expansion. Since $1\le a\le p-1$, one has

$$|a+2mp|\ge1\qquad(m\in\mathbb Z).$$

After $q$ differentiations, each summand is a finite linear combination of terms of the form

$$t^{-A}(a+2mp)^B e^{-(a+2mp)^2/t}.$$

The Gaussian decay in $m$, together with the standard estimate

$$e^{-c/t}=O(t^N)\qquad(t\downarrow0)$$
for every $c>0$ and $N\ge0$, therefore gives

$$\varphi_p^{(q)}(t)=O(t^N)\qquad(t\downarrow0)$$
for all $q,N\ge0$. Thus every derivative tends to $0$ to infinite order at the origin, so the extension by zero is smooth there. Combined with the exponential decay of all derivatives as $t\to\infty$, this shows that the extension by zero is a Schwartz function on $\mathbb R$.

\end{proof}

\subsection{The Laplace transform}

The relation between $\varphi_p$ and $f_p$ is more transparent through the
Dirichlet resolvent rather than by termwise integration of the theta series because
the Dirichlet resolvent kernel for $\lambda-\partial_y^2$ on $(0,p)$ is built from $\sinh(\sqrt\lambda\,y)$
and $\sinh(\sqrt\lambda\,(p-y))$, and that gives the hyperbolic sines defining $h_p$.


We recall the notation, $f_p(x):=h_p(2\sqrt{x})$ where $h_p(y):=\sum_{j=1}^{p-1}\frac{\sinh(jy)}{\sinh(py)}$.
\begin{remark}
In this section $x$ denotes the Laplace transform variable, always with
$\Re x>0$. The first space slot of $K_p$ in \eqref{eq:heat-kernel-spectral} is
occupied throughout by an integer $a\in\{1,\dots,p-1\}$, so the two uses of $x$
never occur together.
\end{remark} 

\begin{proposition}\label{prop:laplace-varphi}
For $\Re x>0$,
\begin{equation}\label{eq:laplace-varphi}
f_p(x)=\int_0^\infty e^{-xt}\varphi_p(t)\dd t.
\end{equation}
In particular,
\begin{equation}\label{eq:varphi-mass}
\int_0^\infty\varphi_p(t)\dd t=f_p(0)=\frac{p-1}{2}.
\end{equation}
\end{proposition}

\begin{proof}
We first take $x>0$.  Let
\[
R_\lambda(a,y):=\int_0^\infty e^{-\lambda s}K_p(a,y,s)\dd s
\]
be the Dirichlet resolvent kernel for $-\partial_y^2+\lambda$ on $(0,p)$.
For $0<y\leq a$,
\begin{equation}\label{eq:resolvent-kernel}
R_\lambda(a,y)=
\frac{\sinh(\sqrt\lambda\,y)
      \sinh(\sqrt\lambda\,(p-a))}
     {\sqrt\lambda\,\sinh(\sqrt\lambda\,p)}.
\end{equation}
Therefore
\begin{equation}\label{eq:resolvent-derivative}
\int_0^\infty e^{-\lambda s}\partial_yK_p(a,0,s)\dd s
 =\frac{\sinh(\sqrt\lambda\,(p-a))}
        {\sinh(\sqrt\lambda\,p)}.
\end{equation}

Using $t=4s$ in \eqref{eq:def-varphi} and then
\eqref{eq:resolvent-derivative} with $\lambda=4x$, we obtain
\begin{align*}
\int_0^\infty e^{-xt}\varphi_p(t)\dd t
 &=\sum_{a=1}^{p-1}
   \int_0^\infty e^{-4xs}\partial_yK_p(a,0,s)\dd s\\
 &=\sum_{a=1}^{p-1}
   \frac{\sinh(2(p-a)\sqrt{x})}{\sinh(2p\sqrt{x})}\\
 &=h_p(2\sqrt{x})=f_p(x).
\end{align*}
Both sides are analytic for $\Re x>0$, so the identity extends there by
analytic continuation.  Letting $x\downarrow0$ gives
\eqref{eq:varphi-mass}.
\end{proof}

\section{Renormalized convolutions and positivity}\label{sec:convolution}

We are going to show that the smooth function $t^{-\delta}\omega(t)$ is the density function of the limit of measures defined by $\varphi_{p, s}(t)$, a rescaled kernel of $\varphi_p$ in Section \ref{sec:theta}.

\subsection{Finite convolutions}

For $s\in\Z$, define the rescaled kernel
\begin{equation}\label{eq:def-varphi-s}
\varphi_{p,s}(t):=p^{-2s}\varphi_p(p^{-2s}t).
\end{equation}
Proposition~\ref{prop:laplace-varphi} and a change of variables give
\begin{equation}\label{eq:scaled-laplace}
\int_0^\infty e^{-xt}\varphi_{p,s}(t)\dd t
 =f_p(p^{2s}x).
\end{equation}

If $I\subset\Z$ is finite, define $\Psi_{p,I}$ by the measure identity
\begin{equation}\label{eq:def-Psi-I}
\mathop{*}_{s\in I}
\bigl(\delta_0+\varphi_{p,s}(t)\dd t\bigr)
 =\delta_0+\Psi_{p,I}(t)\dd t.
\end{equation}
Equivalently, $\Psi_{p,I}$ is the sum of all nonempty convolutions of the
kernels $\varphi_{p,s}$ with $s\in I$.  It is a nonnegative Schwartz function
on $(0,\infty)$.

For $r\geq0$, set
\begin{equation}\label{eq:def-psi0-r}
\psi^{(0)}_{p,r}(t):=\Lambda^{-r}\Psi_{p,[-r,r]}(t).
\end{equation}

\begin{proposition}\label{prop:finite-convolution-transform}
For $\Re x>0$,
\begin{equation}\label{eq:finite-convolution-transform}
\widehat{\psi^{(0)}_{p,r}}(x)
 :=\int_0^\infty e^{-xt}\psi^{(0)}_{p,r}(t)\dd t
 =\Lambda^{-r}\prod_{s=-r}^{r}
  \bigl(1+f_p(p^{2s}x)\bigr)-\Lambda^{-r}.
\end{equation}
Moreover, locally uniformly for $x>0$,
\begin{equation}\label{eq:psi0-transform-limit}
\widehat{\psi^{(0)}_{p,r}}(x)
 \longrightarrow
 x^{-\alpha}\Omega\!\left(\log_{p^2}\frac1x\right).
\end{equation}
\end{proposition}

\begin{proof}
Equation \eqref{eq:finite-convolution-transform} follows immediately from
\eqref{eq:scaled-laplace} and the fact that convolution becomes multiplication
under the Laplace transform.

For the limit, put
\[
T_r(x):=\prod_{s>r}\bigl(1+f_p(p^{2s}x)\bigr).
\]
Then
\[
\prod_{s=-r}^{r}\bigl(1+f_p(p^{2s}x)\bigr)
 =F_p(p^{-2r}x)T_r(x)^{-1}.
\]
By \eqref{eq:hp-forms} and the identity

$$
f_p(x)=h_p(2\sqrt{x}),
$$

we have

$$
0\le f_p(p^{2s}x)
=h_p(2p^s\sqrt{x})
\le (p-1)e^{-2p^s\sqrt{x}}.
$$

Hence, for $x\in[a,b]\subset(0,\infty)$,

$$
0\le f_p(p^{2s}x)
\le (p-1)e^{-2\sqrt{a}\,p^s}
=O\!\left(e^{-c_a p^s}\right),
\qquad c_a:=2\sqrt{a},
$$

uniformly in $x$. Therefore

$$
0\le \log T_r(x)
=\sum_{s>r}\log\!\bigl(1+f_p(p^{2s}x)\bigr)
\le \sum_{s>r}f_p(p^{2s}x)
\le (p-1)\sum_{s>r}e^{-2\sqrt{a}\,p^s}.
$$

The final tail tends to $0$ as $r\to\infty$, independently of $x\in[a,b]$. Thus

$$
T_r(x)\longrightarrow 1
$$

uniformly on $[a,b]$.
  On the other hand,
Proposition~\ref{prop:F-periodic} gives
\[
F_p(p^{-2r}x)
 =p^{2\alpha r}x^{-\alpha}
  \Omega\!\left(r+\log_{p^2}\frac1x\right)(1+o(1)).
\]
Since $p^{2\alpha}=\Lambda$ and $\Omega$ has period $1$,
\[
\Lambda^{-r}F_p(p^{-2r}x)
 \longrightarrow
 x^{-\alpha}\Omega\!\left(\log_{p^2}\frac1x\right)
\]
locally uniformly.  The subtracted term $\Lambda^{-r}$ tends to zero, proving
\eqref{eq:psi0-transform-limit}.
\end{proof}

To match the normalization in the coefficient asymptotic, define
\begin{equation}\label{eq:def-psi-r}
\psi_{p,r}(t):=4\psi^{(0)}_{p,r}(2t)
 =4\Lambda^{-r}\Psi_{p,[-r,r]}(2t).
\end{equation}
Then
\[
\widehat{\psi_{p,r}}(x)=2\widehat{\psi^{(0)}_{p,r}}(x/2).
\]
Define the periodic function
\begin{equation}\label{eq:def-Omega-B}
\Omega_B(u):=2^{1+\alpha}
\Omega\bigl(u+\log_{p^2}2\bigr).
\end{equation}
If $\Omega(u)=\sum_mc_me^{2\pi imu}$, then
\begin{equation}\label{eq:Omega-B-Fourier}
\Omega_B(u)=\sum_{m\in\Z}d_me^{2\pi imu},
\end{equation}
with the same $d_m$ as in \eqref{eq:def-dm}.

\begin{corollary}\label{cor:psi-transform-limit}
Locally uniformly for $x>0$,
\begin{equation}\label{eq:psi-transform-limit}
\widehat{\psi_{p,r}}(x)
 \longrightarrow
\widehat\psi(x):=
 x^{-\alpha}\Omega_B\!\left(\log_{p^2}\frac1x\right).
\end{equation}
\end{corollary}

\begin{proof}
Apply \eqref{eq:psi0-transform-limit} at $x/2$ and multiply by $2$:
\[
2\left(\frac x2\right)^{-\alpha}
\Omega\!\left(\log_{p^2}\frac2x\right)
 =x^{-\alpha}\Omega_B\!\left(\log_{p^2}\frac1x\right).
\]
\end{proof}

\subsection{The density function of the limit measure}

Let
\begin{equation}\label{eq:def-mu-r}
\mu_r(\dd t):=\psi_{p,r}(t)\dd t.
\end{equation}

\begin{proposition}\label{prop:vague-limit}
The measures $\mu_r$ converge vaguely on $[0,\infty)$ to a positive Radon
measure $\mu$.  Its Laplace transform is $\widehat\psi$ from
\eqref{eq:psi-transform-limit}.  Moreover, $\mu$ has the smooth density
\begin{equation}\label{eq:def-psi-density}
\psi(t):=
\sum_{m\in\Z}\frac{d_m}{\Gamma(\beta_m)}t^{\beta_m-1},
\qquad t>0.
\end{equation}
Consequently,
\begin{equation}\label{eq:psi-omega}
\psi(t)=t^{\alpha-1}\omega(t)=t^{-\delta}\omega(t).
\end{equation}
\end{proposition}

\begin{proof}

 Fix $x_0>0$ and $T>0$. For any $0<t\le T$, we have $e^{-x_0T} \le e^{-x_0t}$. Hence, $1 \le e^{-x_0T}e^{x_0t}$. As $\psi_{p, r}$ is nonnegative, it follows that $\mu_r([0,T])
 \leq e^{x_0T}\widehat{\psi_{p,r}}(x_0)$. Since $\widehat{\psi_{p,r}} \to \widehat{\psi}$ by \eqref{eq:psi-transform-limit}, $\mu_r([0,T])<C_T$ for some constant $C_T$ independent of $r$.

Since, for every $T>0$,

$$
\sup_r \mu_r([0,T])<\infty,
$$

the restrictions $\mu_r|_{[0,T]}$ form a bounded family of finite positive
measures on the compact interval $[0,T]$.  By weak compactness of bounded
sets of finite measures, every sequence $(\mu_r)$ has a subsequence converging
weakly on $[0,T]$.

Applying this successively for $T=1,2,3,\ldots$, and at each stage extracting
a subsequence of the preceding one, a diagonal argument produces a single
subsequence $(\mu_{r_j})$ that converges weakly on every compact interval
$[0,T]$.  The resulting limits on the intervals $[0,T]$ are compatible and
therefore define a positive Radon measure $\mu$ on $[0,\infty)$.

Equivalently,

$$
\int_0^\infty \phi(t)\,\mu_{r_j}(\dd t)
\longrightarrow
\int_0^\infty \phi(t)\,\mu(\dd t)
$$

for every $\phi\in C_c([0,\infty))$.  This is precisely the statement that
$\mu_{r_j}\to\mu$ vaguely on $[0,\infty)$.

A positive Radon measure on $[0,\infty)$ is uniquely determined by its Laplace transform; see \cite[XIII.1. Theorem 1]{Feller71}. Hence every vaguely convergent subsequence of $(\mu_r)$ has the same limit $\mu$. Since the family is relatively compact for vague convergence, it follows that the full sequence satisfies

$$
\mu_r\longrightarrow\mu
$$

vaguely on $[0,\infty)$.
Let us fix $x>0$. Since the function $t\mapsto e^{-xt}$ is not compactly supported, vague convergence cannot be applied to it directly. We therefore first truncate the integral. For $T>0$, choose $\chi_T\in C_c([0,\infty))$ such that

$$
0\le \chi_T\le 1,
\qquad
\chi_T(t)=1 \quad \text{for } 0\le t\le T.
$$

Then vague convergence gives

$$
\int_0^\infty e^{-xt}\chi_T(t)\,\mu_{r_j}(\dd t)
\longrightarrow
\int_0^\infty e^{-xt}\chi_T(t)\,\mu(\dd t).
$$

On the other hand, for $t\ge T$,

$$
e^{-xt}
=e^{-xt/2}e^{-xt/2}
\le e^{-xT/2}e^{-xt/2},
$$

and hence

$$
\int_T^\infty e^{-xt}\mu_r(\dd t)
\le e^{-xT/2}\widehat{\psi_{p,r}}(x/2),
$$
and the same bound passes to the vague limit $\mu$.
By \eqref{eq:psi-transform-limit}, the quantities
$\widehat{\psi_{p,r}}(x/2)$ are uniformly bounded in $r$, so

$$
\sup_r\int_T^\infty e^{-xt}\mu_r(\dd t)
\longrightarrow 0
\qquad (T\to\infty).
$$
Therefore, letting first $r\to\infty$ and then $T\to\infty$, we obtain

$$
\int_0^\infty e^{-xt}\mu(\dd t)
=
\lim_{j\to\infty}
\int_0^\infty e^{-xt}\mu_{r}(\dd t)
=
\widehat\psi(x),
$$

where the last equality follows from \eqref{eq:psi-transform-limit}.

By \eqref{eq:Omega-B-Fourier},
\begin{equation}\label{eq:hatpsi-Fourier}
\widehat\psi(x)=\sum_{m\in\Z}d_mx^{-\beta_m}.
\end{equation}
Since $\beta_m=\alpha+i\tau_m$, we have
\[
t^{\beta_m-1}=t^{\alpha-1}t^{i\tau_m}.
\]
Hence \eqref{eq:def-psi-density} is obtained from the series
\eqref{eq:def-omega} by multiplication by the smooth factor
$t^{\alpha-1}$.  The exponential estimate established in
Lemma~\ref{lem:periodic-transfer} therefore implies that the series
\eqref{eq:def-psi-density}, together with all derivatives with respect to
$\log t$, converges absolutely and locally uniformly on $(0,\infty)$.  Also,
\[
\int_0^\infty e^{-xt}t^{\beta_m-1}\dd t
 =\Gamma(\beta_m)x^{-\beta_m}.
\]
Absolute summability permits termwise integration, so the smooth function
\eqref{eq:def-psi-density} has Laplace transform \eqref{eq:hatpsi-Fourier}.
Uniqueness of Laplace transforms now gives $\mu(\dd t)=\psi(t)\dd t$ on
$(0,\infty)$.  Finally, comparing \eqref{eq:def-psi-density} with
\eqref{eq:def-omega} gives \eqref{eq:psi-omega}.
\end{proof}

\subsection{Strict positivity}

The positivity of every finite convolution shows that $\psi\geq0$ almost
everywhere, but Theorem~\ref{thm:main} requires strict pointwise positivity. Before proving the strict positivity, we will first prove three lemmas.

Recall that the measure $\mu= \psi(t)\dd t.$
\begin{lemma} \label{lemma:mu_as_convolution}
    We have $\mu=(\delta_0+\chi(t) \dd t) * \nu$ where $\chi(t)=2\varphi_p(2t)$ and  $\nu$ is a positive Radon measure whose Laplace transform is $\widehat\psi(x)/({1+f_p(x/2)}).$
\end{lemma}

\begin{proof}
Recall that $\psi_{p, r}(t)=4\psi_{p, r}^{(0)}(2t).$ By the scaling property of Laplace transform and Proposition \ref{prop:finite-convolution-transform}, $$\widehat{\psi_{p,r}}(x)
 :=\int_0^\infty e^{-xt}\psi_{p,r}(t)\dd t
 =2\Lambda^{-r}\prod_{s=-r}^{r}
  \bigl(1+f_p(p^{2s}x/2)\bigr)-2\Lambda^{-r}.$$

For $r\geq1$, define
\[
L_r(x):=
2\Lambda^{-r}
\prod_{\substack{-r\leq s\leq r\\ s\neq0}}
\left(1+f_p(p^{2s}x/2)\right).
\]

Since 
\[
\begin{aligned}
\bigl(1+f_p(x/2)\bigr)L_r(x)
&=
2\Lambda^{-r}
\prod_{s=-r}^{r}
\left(1+f_p(p^{2s}x/2)\right) \\
&=
\widehat{\psi_{p,r}}(x)+2\Lambda^{-r}.
\end{aligned}
\] and $\Lambda>1$, Corollary~\ref{cor:psi-transform-limit} yields, locally
uniformly for $x>0$,
\begin{equation}\label{eq:def-L-limit}
L_r(x)\longrightarrow
L(x):=
\frac{\widehat\psi(x)}{1+f_p(x/2)}.
\end{equation}

We next show that $L(x)$ is the Laplace transform of a positive Radon measure.

First, there exists a positive measure, say $\nu_r$, whose Laplace transform is $L_r(x)$ since each factor $1+f_p(p^{2s}x/2)$ of $L_r(x)$ is the Laplace transform of a suitable rescaling of
$\delta_0+\varphi_p(t)\,\dd t$ and multiplication of Laplace transforms
corresponds to convolution of measures. 

For fixed $x_0,T>0$, we have $\nu_r([0,T])
\leq e^{x_0T}L_r(x_0),
$ and the right-hand side is bounded independently of $r$ by
\eqref{eq:def-L-limit}.  Thus $(\nu_r)$ is locally bounded in mass. The compactness-and-tail argument used in the proof of
Proposition~\ref{prop:vague-limit} now applies verbatim. Namely, after
passing to a subsequence, $\nu_r$ converges vaguely to a positive Radon
measure $\nu$, while
\[
\int_T^\infty e^{-xt}\nu_r(\dd t)
\leq e^{-xT/2}L_r(x/2)
\]
shows that the Laplace tails are uniformly negligible.  Consequently,
for every $x>0$,
\[
\int_{[0,\infty)}e^{-xt}\nu(\dd t)
=
\lim_{r\to\infty}L_r(x)
=
L(x).
\]
Hence $L$ is the Laplace transform of $\nu$.

Now set $\chi(t)=2\varphi_p(2t)$. Since $(1+f_p(x/2))L(x)=\widehat{\psi}(x)$, we have that the Laplace transform of $(\delta_0+\chi(t)\dd t) * \nu$ is $\widehat{\psi}(x)$. Uniqueness of Laplace
transforms therefore gives the measure identity
\begin{equation}\label{eq:mu-factor-measure}
\psi(t)\,\dd t
=
\bigl(\delta_0+\chi(t)\,\dd t\bigr)*\nu.
\end{equation}
\end{proof}

Since $\mu=(\delta_0+\chi(t) \dd t) * \nu$ and the convolution of a measure with $\delta_0$ is the original measure, we have $\mu=\nu+ \chi(t) \dd t * \nu.$ The next two lemmas treat the two terms of $\mu$ separately. 
\begin{lemma} \label{lemma:nu>0}
    For any $t>0$, $\nu([0,t))>0$.
\end{lemma}

\begin{proof}
We prove by contradiction. Suppose that $\nu([0,t_0))=0$ for some $t_0>0$.  Then $\nu$ is
supported in $[t_0,\infty)$, so 
\[
\begin{aligned}
L(x)
&=
\int_{[t_0,\infty)}e^{-xt}\nu(\dd t) \\
&=
\int_{[t_0,\infty)}
e^{-xt/2}e^{-xt/2}\nu(\dd t) \\
&\leq
e^{-xt_0/2}
\int_{[t_0,\infty)}e^{-xt/2}\nu(\dd t) \\
&=
e^{-xt_0/2}L(x/2).
\end{aligned}
\]

Now, we claim that there exist $c, C>0$ such that $cx^{-\alpha}\le L(x) \le Cx^{-\alpha}$ for sufficiently large $x$.

Suppose the claim holds. Then for all sufficiently large $x$, since $c x^{-\alpha}
\leq L(x)$ and $L(x/2)
\leq C(x/2)^{-\alpha}
=
C2^\alpha x^{-\alpha}$, it follows that $$c
\leq
C2^\alpha e^{-xt_0/2},$$
which is impossible as $x\to\infty$. Hence, $\nu([0,t))>0$ for any $t>0.$

To prove the claim, recall that $$L(x)=
\frac{\widehat\psi(x)}{1+f_p(x/2)}.$$
    
 By Proposition~\ref{prop:F-periodic}, the periodic function $\Omega$ is
bounded above and below by positive constants on the real axis.  Hence the
same is true of $\Omega_B$ by \eqref{eq:def-Omega-B}.  Since
\[
\widehat\psi(x)
=
x^{-\alpha}
\Omega_B\!\left(\log_{p^2}\frac1x\right),
\]
there are constants $d,D>0$ such that for any $x>0$, 
$d x^{-\alpha}
\leq
\widehat\psi(x)
\leq
D x^{-\alpha}.$

Since $f_p(x/2)\to0$ as $x\to\infty$, we may take $c=d/(1+\epsilon)$ for some $\epsilon>0$ and $C=D$. Then the claim follows.

\end{proof}

\begin{lemma}\label{lemma:h_positivity}
    The measure $\chi(t)\dd t * \nu=h(t)\dd t$ for some continuous function $h(t)$. Moreover, $h(t)>0$ on  for any $t>0$. 
\end{lemma}

\begin{proof}
    First, recall that $\chi(t)=2\varphi_p(2t)$.. Hence, by Proposition \ref{prop:varphi-properties}, $\chi(t)>0$ for $t>0$ and after extended by zero, $\chi(t)$ is still a continuous function.    
    
    Now define 
\begin{equation}\label{eq:def-h-positive}
h(t)=
\int_{[0,t)}\chi(t-s)\nu(\dd s),
\qquad t>0.
\end{equation}
For every nonnegative compactly supported function $g$,
\[
\begin{aligned}
\int_0^\infty g(t)
\bigl((\chi(u)\,\dd u)*\nu\bigr)(\dd t)
&=
\int_{[0,\infty)}
\int_0^\infty
g(s+u)\chi(u)\,\dd u\,\nu(\dd s) \\
&=
\int_0^\infty
g(t)
\left(
\int_{[0,t]}
\chi(t-s)\nu(\dd s)
\right)\dd t.
\end{aligned}
\]
Since $\chi(0)=0$, the intervals $[0,t]$ and $[0,t)$ give the same
integral. Hence, $\chi(t)\dd t * \nu=h(t)\dd t$.

To show $h$ is continuous, for $t\geq0$ we may write
\[
h(t)
=
\int_{[0,\infty)}\chi(t-s)\nu(\dd s),
\]
where $\chi(u)=0$ for $u\leq0$.  Fix $T>0$.  If $t,t'\in[0,T]$, only
$s\in[0,T]$ contributes, and hence
\[
\begin{aligned}
|h(t)-h(t')|
&\leq
\int_{[0,T]}
|\chi(t-s)-\chi(t'-s)|\,\nu(\dd s) \\
&\leq
\nu([0,T])
\sup_{s\in[0,T]}
|\chi(t-s)-\chi(t'-s)|.
\end{aligned}
\]
Because $\chi$ is uniformly continuous on the compact interval $[-T,T]$,
the last expression tends to $0$ as $t'\to t$.  Thus $h$ is continuous.

Now we want to show that $h(t)>0$ for any $t>0$.  Fix $t>0$.  By
Lemma \ref{lemma:nu>0}, we have $\nu([0,t))>0.$ For every $s\in[0,t)$ we have $t-s>0$, and therefore $\chi(t-s)>0$.
Hence
\[
h(t)
=
\int_{[0,t)}\chi(t-s)\nu(\dd s)
>0.
\]
\end{proof}

\begin{proposition}\label{prop:omega-positive}
The functions $\psi$ and $\omega$ are strictly positive on $(0,\infty)$.  In
particular,
\begin{equation}\label{eq:omega-positive-min}
\min_{1\leq t\leq p^2}\omega(t)>0.
\end{equation}
\end{proposition}

\begin{proof}

By Lemma \ref{lemma:mu_as_convolution}, we have \[
\psi(t)\,\dd t
=
\nu+h(t)\,\dd t.
\]
Since $\nu$ is a positive measure, it follows that $\psi(t)\geq h(t)$
for almost every $t>0$.  As $\psi$ and $h$ are continuous, the
inequality holds everywhere; therefore, $\psi(t)\geq h(t)$ for all $t>0$. By Lemma \ref{lemma:h_positivity}, $h(t)>0$ on $(0,\infty)$. Hence $\psi$ is strictly positive on $(0,\infty)$. Finally, \eqref{eq:psi-omega} gives $\omega(t)=t^\delta\psi(t)>0$ for all $t>0$.

By Lemma~\ref{lem:periodic-transfer}, $\omega$ is continuous and
multiplicatively $p^2$-periodic.  Hence, it attains a minimum on the
compact interval $[1,p^2]$ with a positive value. 
\end{proof}

\begin{proof}[Proof of Theorem~\ref{thm:main}]
Proposition~\ref{prop:additive-asymptotic} gives,
\[
2^{-k}b_k
 =k^{-\delta}\omega(k)+O(k^{-\delta-1/2}).
\]
Lemma~\ref{lem:periodic-transfer} shows that $\omega$ is smooth in
$\log t$ and multiplicatively $p^2$-periodic.  Proposition~\ref{prop:omega-positive}
shows that it is strictly positive and has a positive minimum on one
multiplicative period.  Dividing the additive error by
$k^{-\delta}\omega(k)$ therefore gives
\[
2^{-k}b_k
 =k^{-\delta}\omega(k)\bigl(1+O(k^{-1/2})\bigr),
\]
and multiplication by $2^k$ gives the asserted relative asymptotic.

Proposition~\ref{prop:varphi-properties} proves that 
$\varphi_p$ is strictly positive and Schwartz.  Proposition~\ref{prop:vague-limit}
shows that the renormalized convolution densities
\[
\psi_{p,r}(t)=4\Lambda^{-r}\Psi_{p,[-r,r]}(2t)
\]
converge vaguely to
\[
\psi(t)=t^{-\delta}\omega(t).
\]
This is the convolution realization stated in the theorem.
\end{proof}

\printbibliography

\end{document}